\documentclass[11pt]{amsart}

\usepackage{amsmath,amsthm,amscd,amssymb,subcaption}
\usepackage[letterpaper,left=1.5in,right=1.5in,top=1.25in,bottom=1.25in]{geometry}
\usepackage{pb-diagram}
\usepackage{graphicx, xcolor}

\newtheorem{thm}{Theorem}[section]
\newtheorem{lem}[thm]{Lemma}
\newtheorem{prop}[thm]{Proposition}
\newtheorem{cor}[thm]{Corollary}

\theoremstyle{definition}
\newtheorem{example}[thm]{Example}
\newtheorem{defn}[thm]{Definition}
\newtheorem{rem}[thm]{Remark}

\numberwithin{equation}{section}
\newcommand{\C}{\mathcal{C}}
\newcommand{\lk}{{\mathrm{\ell k}}}

\newcommand{\tb}{\text{tb}}
\newcommand{\TB}{\text{TB}}
\newcommand{\rot}{\text{rot}}
\newcommand{\s}{\text{s}\,}
\newcommand{\g}{\text{g}}
\newcommand{\gsh}{\text{g}_\text{sh}}
\newcommand{\Arf}{\text{Arf}}
\newcommand{\Wh}{\text{Wh}}
\newcommand{\K}{\mathcal{K}}

\title[Shake slice and shake concordant knots]{Shake slice and shake concordant knots}

\author{Tim D.\ Cochran$^{\dag}$}
\address{Department of Mathematics MS-136, P.O. Box 1892, Rice University, Houston, TX 77251-1892}
\email{cochran@rice.edu}
\author{Arunima Ray$^{\dag\dag}$}
\address{Department of Mathematics MS-050, Brandeis University, 415 South St., Waltham, MA 02453}
\email{aruray@brandeis.edu}

\thanks{$^{\dag}$Partially supported by the National Science Foundation  DMS-1309081, and the Simons Foundation}
\thanks{$^{\dag\dag}$Partially supported by NSF--DMS--1309081 and the Nettie S.\ Autrey Fellowship (Rice University)}


\begin{document}
%\linenumbers

\date{\today}
\begin{abstract}  
A crucial step in the surgery-theoretic program to classify smooth manifolds is that of representing a middle--dimensional homology class by a smoothly embedded sphere. This step fails even for the simple $4$--manifolds obtained from the 4--ball by adding a 2--handle with framing $r$ along some knot $K\hookrightarrow \partial B^4$. An \textit{$r$--shake slice} knot is one for which a generator of the second homology of this $4$--manifold can be represented by a smoothly embedded $2$--sphere. It is not known whether there exist 0--shake slice knots that are not slice. We define a relative notion of shake sliceness of knots, which we call \textit{shake concordance} which is easily seen to be a generalization of classical concordance, and we give the first examples of knots that are 0--shake concordant but not concordant; these may be chosen to be topologically slice. Additionally, for each $r$ we completely characterize $r$--shake slice and $r$--shake concordant  knots in terms of concordance and satellite operators. Our characterization allows us to construct new families of possible $r$--shake slice knots that are not slice. 
\end{abstract}

\maketitle
%===========================================================

\section{Introduction}\label{sec:intro}

Given a homology class in a manifold, it is often of interest to find submanifolds representing that class. For example, a key step in the surgery-theoretic program to classify smooth manifolds is that of representing a middle--dimensional homology class by a smoothly embedded sphere. This step fails for 4--manifolds. The simplest examples of this failure arise as follows. Let $K$ be an oriented knot in $S^3$ and $r$ an integer; let $W_K^r$ be the 4--manifold obtained by adding a 2--handle to $B^4$ along $K$ with framing $r$. We say that $K$ is \textit{$r$--shake slice} -- a notion introduced in \cite{Ak77} -- if there exists a smoothly embedded 2--sphere in $W_K^r$ representing a generator of $H_2(W_K^r)\cong \mathbb{Z}$. Not every such homology class can be represented by an embedded sphere, that is, not all knots are $r$--shake slice since, for example, certain knot signatures  are obstructions~\cite{Ak77}.

There is also a relative version of this notion (mentioned but not defined in Kirby's problem list ~\cite[p.\ 515]{Kirbylist1984}). Given oriented knots $K_i\hookrightarrow S^3\times \{i\}$, $i=0,1$, let $W_{K_0,\,K_1}^r$ denote  the 4--manifold obtained by adding two 2--handles to $S^3\times [0,1]$ along the $K_i$, each with framing $r$. Then $K_0$ is \textit{$r$--shake concordant} to $K_1$ if there exists a 2--sphere, smoothly embedded in $W^r_{K_0,\,K_1}$, representing the $(1,\,1)$ class of $H_2(W^r_{K_0,\,K_1})$ (a precise definition is given in Section~\ref{sec:background}).  Surprisingly little work  (summarized below) has been done on these notions. Instead, research has focused on a special case, introduced by Fox and Milnor: the knot $K$ is \textit{slice} if there exists a  2--disk, smoothly embedded in $B^4$, whose boundary is $K$; and the knots $K_i\hookrightarrow S^3\times \{i\}$, $i=0,1$, are \textit{concordant} if there is an annulus, smoothly embedded in $S^3\times [0,1]$, which restricts on its boundary to the $K_i$ ~\cite{FoMi57,FoMi66}. The main goal of this present paper is to investigate the difference between the notions of shake concordance and concordance. In particular we give the first examples of 0--shake concordant knots that are not concordant, and we completely characterize $r$--shake slice and $r$--shake concordant knots in terms of concordance, for all integers $r$.

We briefly review what was previously known on this subject.  Clearly any slice knot is $r$--shake slice and any two concordant knots are $r$--shake concordant, for any $r$. In his seminal paper~\cite{Ak77}, Akbulut gave an example of a 1--shake slice knot and a 2--shake slice knot neither of which are slice (Lickorish gave a different construction later in \cite{Lic79}). Boyer gave many examples of what he called `pseudo-1--shake slice' knots (allowing himself to alter the smooth structure on $W_K^r$) in \cite{Boy83, Boy85}. More recently, in~\cite{Ak93}, Akbulut generalized the 1--shake slice examples in~\cite{Ak77}, and constructed $r$--shake slice knots for each \textit{non-zero} $r$, which are not always slice. Building on \cite{Ak93, Om11}, Abe--Jong--Omae--Takeuchi \cite{AbeJongOmaeTake13} have given examples, also for each non-zero $r$, of $r$--shake slice knots that are not slice. It is easy to see that any $r$--shake slice knot is $r$--shake concordant to the unknot, and so the above examples also give examples, for $r\neq 0$, of knots that are $r$--shake concordant but not concordant.  However, it is fair to say that there is no systematic understanding of $r$--shake concordance. Moreover, nothing is known about the important case $r=0$.  

Differentiating between concordance and 0--shake concordance, and between slice knots and 0--shake slice knots, is made difficult by the following elementary fact. Let $M^r_K$ denote the 3--manifold obtained by performing $r$--framed surgery on $S^3$ along the knot $K$. 

\newtheorem*{prop_1}{Proposition~\ref{prop:shakeconchomcob}}
\begin{prop_1} If $K$ is 0--shake concordant to $J$ then $M_K^0$ is homology cobordant to $M_J^0$ preserving the homology class of the (positive) meridian. Consequently, if $K$ is 0--shake slice then $K$ is slice (i.e.\ bounds a smoothly embedded disk) in some homology $B^4$. \end{prop_1}

Recall that the classical concordance invariants are determined by the homology cobordism type of $M_K^0$, along with the homology class of the positive meridian. Hence the question of whether there exist 0--shake slice knots that are not slice is especially difficult because there are presently no known invariants that can distinguish between a knot being slice in $B^4$ and it being slice in merely a homology $B^4$, except (possibly) Rasmussen's $\s$--invariant and its generalizations.  

Therefore it may be surprising that here we find success in the relative case.

\newtheorem*{thm:examples}{Theorem~\ref{thm:examples}}
\begin{thm:examples} For any integer $r$, there exist infinitely many knots which are distinct in smooth concordance but are pairwise $r$--shake concordant. For $r=0$, there exist topologically slice knots with this property as well. 

In addition, for any integer $r$, none of $\tau$, $\s$, or slice genus is invariant under $r$--shake concordance.
\end{thm:examples}

This result can be seen as a consequence of our complete characterization of $r$--shake concordance in terms of concordance and certain winding number one satellites. Let $P$ be a pattern knot, i.e.\ a knot inside a solid torus (an example is shown in Figure~\ref{fig:P}). For any knot $K$, let $P_r(K)$ denote the $r$--twisted satellite of $K$ with pattern $P$~\cite[p.\ 10]{Lick97}\cite[p.\ 110]{Ro90}. $P(K)$ denotes the 0--twisted (i.e.\ classical) satellite of $K$ with pattern $P$. Let $\widetilde{P}$ denote the knot in $S^3$ given by $P$ when the solid torus is placed in $S^3$ in the standard unknotted manner, or equivalently, $\widetilde{P}=P(U)$ where $U$ is the unknot. 

\begin{figure}[t]
\centering
\includegraphics[width=2.5in]{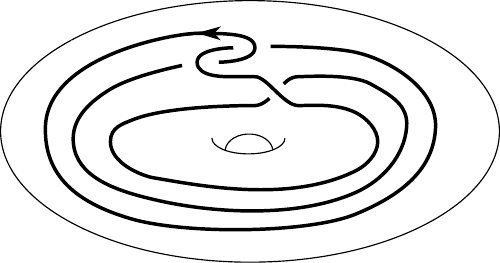}
\caption{A winding number one pattern $P$ where $\widetilde{P}$ is unknotted. This particular pattern will be referred to as the Mazur pattern, inducing the Mazur satellite operator.}\label{fig:P}
\end{figure}

\newtheorem*{thm:main}{Theorem~\ref{thm:main}}\begin{thm:main} For any integer $r$, the knots $K$ and $J$ are $r$--shake concordant if and only if there exist winding number one patterns $P$ and $Q$, with $\widetilde{P}$ and $\widetilde{Q}$ ribbon knots, such that $P_r(K)$ is concordant to $Q_r(J)$. 
\end{thm:main}

\newtheorem*{cor:sameequivrelation}{Corollary~\ref{cor:sameequivrelation}}\begin{cor:sameequivrelation}  For any integer $r$, the equivalence relation on the set of knots generated by $r$--shake concordance is the same as that generated by concordance together with the relation $K\sim P_r(K)$ for all $K$ and all winding number one patterns $P$ with $\widetilde{P}$ a ribbon knot.
\end{cor:sameequivrelation}

From Theorem~\ref{thm:main} we also obtain a characterization of $r$--shake slice knots as follows. 

\newtheorem*{cor:characterizeshakeslice}{Corollary~\ref{cor:characterizeshakeslice}}\begin{cor:characterizeshakeslice} For any integer $r$, a knot $K$ is $r$--shake slice if and only if there exists a winding number one pattern $P$, with $\widetilde{P}$ ribbon, such that $P_r(K)$ is slice.\end{cor:characterizeshakeslice}

It is easy to see that any pattern knot induces, for each value of $r$, a \textit{satellite operator}, $$P_r:\C \to \C,$$ on the set of knot concordance classes $\C$. If $r=0$ the subscript is often suppressed. Some of our results are especially interesting in light of recent research on the injectivity and surjectivity of such satellite operators  \cite{CDR14, CHL11, DR13, Lev14}.  For example, reinterpreting Corollary~\ref{cor:characterizeshakeslice} in this language for $r=0$ yields the following.

\newtheorem*{cor:characterizezero2}{Corollary~\ref{cor:characterizezero2}}\begin{cor:characterizezero2} There exists a $0$--shake slice knot that is not a slice knot if and only if there exists some winding number one satellite operator $P:\C\to\C$, with $\widetilde{P}$ ribbon, which fails to be weakly injective, i.e. there exists a knot $K\neq 0$ (not slice) such that $P(K)=0$ (is slice).\end{cor:characterizezero2}

Compare this to the result from \cite{CDR14} that for any winding number one $P$ with $\widetilde{P}$ slice, the induced operator $P:\C^*\to\C^*$  is injective (here $\C^*$ is concordance in a homology $S^3\times [0,1]$); which implies that if  $P(K)$ is slice then $K$ is slice in a homology $4$--ball, agreeing with the conclusion of Proposition~\ref{prop:shakeconchomcob}. Similarly, recall that, for any $r\neq 0$, there \textit{do} exist knots which are $r$--shake slice but not slice, due to \cite{Ak77, AbeJongOmaeTake13}. Thus, for each $r\neq 0$, there must exist a pattern $P$ with $\widetilde{P}$ ribbon, such that for some non-slice $K$, $P_r(K)$ is slice. 

More generally, we also obtain new results regarding the \textit{$r$--shake genus} of a knot. For a knot $K$, the $r$--shake genus of $K$, denoted $\gsh^r(K)$, is the least genus of a smooth connected submanifold representing a generator of $H_2(W_K^r)$. 

Clearly, $\gsh^r(K)\leq \g_4(K)$ for all $r$, where $\g_4$ denotes slice genus. It is an open question whether $\gsh^0(K)$ is equal to $\g_4(K)$ for all $K$ \cite[Problem 1.41(A)]{kirbylist}; this is clearly a generalization of the question of whether all 0--shake slice knots are slice. Note that the previously mentioned work of \cite{Ak77, AbeJongOmaeTake13} shows that the $r$--shake genus can be strictly less than the slice genus for $r\neq 0$.

We establish what we call an  $r$--shake slice--Bennequin inequality, which becomes our main tool.

\newtheorem*{cor:shakeslicebennequin}{Corollary~\ref{cor:shakeslicebennequin}}\begin{cor:shakeslicebennequin} [$r$--shake slice--Bennequin inequality]For any Legendrian representative $\mathcal{K}$ of a knot $K$ with $\tb(\mathcal{K})-1\geq r$, $$\tb(\mathcal{K})+\lvert\rot(\mathcal{K})\rvert\leq 2\gsh^r(K)-1.$$
\end{cor:shakeslicebennequin}

As a result, Thurston--Bennequin numbers can obstruct a knot's being $r$--shake slice and more generally can give a lower bound on the $r$--shake genus. We show that the infinite family of knots from Theorem~\ref{thm:examples} have distinct $r$--shake genera,  implying that the $r$--shake genus is not an invariant of $r$--shake concordance (Corollary~\ref{cor:shakegenusnotinvariant}).

We also obtain the following result, which allows us to restate questions about the $r$--shake genera and slice genus of a knot in terms of the slice genera of its satellites. 

\newtheorem*{prop:shakedecreaseslice}{Proposition~\ref{prop:shakedecreaseslice}}\begin{prop:shakedecreaseslice} Fix an integer $r$ and knot $K$; then $\gsh^r(K)=\g_4(K)$ if and only if $\g_4(P_r(K))\geq \g_4(K)$ for all winding number one patterns $P$ with $\widetilde{P}$ slice. \end{prop:shakedecreaseslice}

There are no known examples of patterns $P$ with $\widetilde{P}$ slice which strictly decrease slice genus, that is, $\g_4(P(K))<\g_4(K)$ for some $K$. See Section \ref{sec:shakegenus} for additional results about shake genus.

Lastly, in Section~\ref{sec:oldexamples} we describe how the previous examples of $r$--shake slice knots given by Akbulut and Abe--Jong--Omae--Takeuchi satisfy our characterization. Our characterization theorem also allows us to construct possible new examples of $r$--shake slice knots, as follows.

\newtheorem*{prop:newshakeslice}{Proposition~\ref{prop:newshakeslice}}\begin{prop:newshakeslice}Let $P$ be a winding number one pattern in a solid torus $V$, such that $\widetilde{P}$ is slice, and the meridian of $P$ is in the subgroup of $\pi_1(V-N(P))$ normally generated by the meridian of $V$. Then for any integer $r$, the knot $P_r(U)$ is $r$--shake slice, where $U$ is the unknot, modulo the smooth 4-dimensional Poincar\'{e} Conjecture.
\end{prop:newshakeslice}

Examples of such knots are given in Figure~\ref{fig:newshakeslice}. Unfortunately, the above proposition does not yield any new knots that are possibly 0--shake slice since if $r=0$, $P_r(U)$ is just the knot $\widetilde{P}$ which is slice by hypothesis.

\subsection*{Remark} Anthony Bosman has extended several of our results to links \cite{Bos15}.

\subsection*{Outline} Section~\ref{sec:background} gives some background and definitions. Section~\ref{sec:characterization} is devoted to proving our characterization theorem, Theorem~\ref{thm:main}, while Section~\ref{sec:examples} constructs the examples mentioned in Theorem~\ref{thm:examples}. Section~\ref{sec:obstructions} is a short section gathering together several obstructions to $r$--shake concordance. In Section~\ref{sec:shakegenus} we describe several properties of $r$--shake genus and prove the $r$--shake slice--Bennequin inequality, Corollary~\ref{cor:shakeslicebennequin}. Lastly, in Section~\ref{sec:oldexamples} we show that the previously known examples of $r$--shake slice knots satisfy our characterization modulo the smooth 4--dimensional Poincar\'{e} Conjecture, and give new examples of possible $r$--shake slice knots.

\subsection*{Acknowledgments} This project started when the second author was in her final year as a PhD student of the first author. The second author is deeply indebted to the first for his careful guidance and constant support, given freely and often, even after she graduated. 

\subsection*{Note} The first author, Tim Cochran, passed away unexpectedly and tragically in December 2014, shortly before this paper was posted on the math arXiv. 

%====================================================
\section{Background and preliminaries}\label{sec:background}

We first review well-known alternative definitions of $r$--shake slice knots and $r$--shake concordance.

\begin{figure}[b]
\centering
\includegraphics[width=2in]{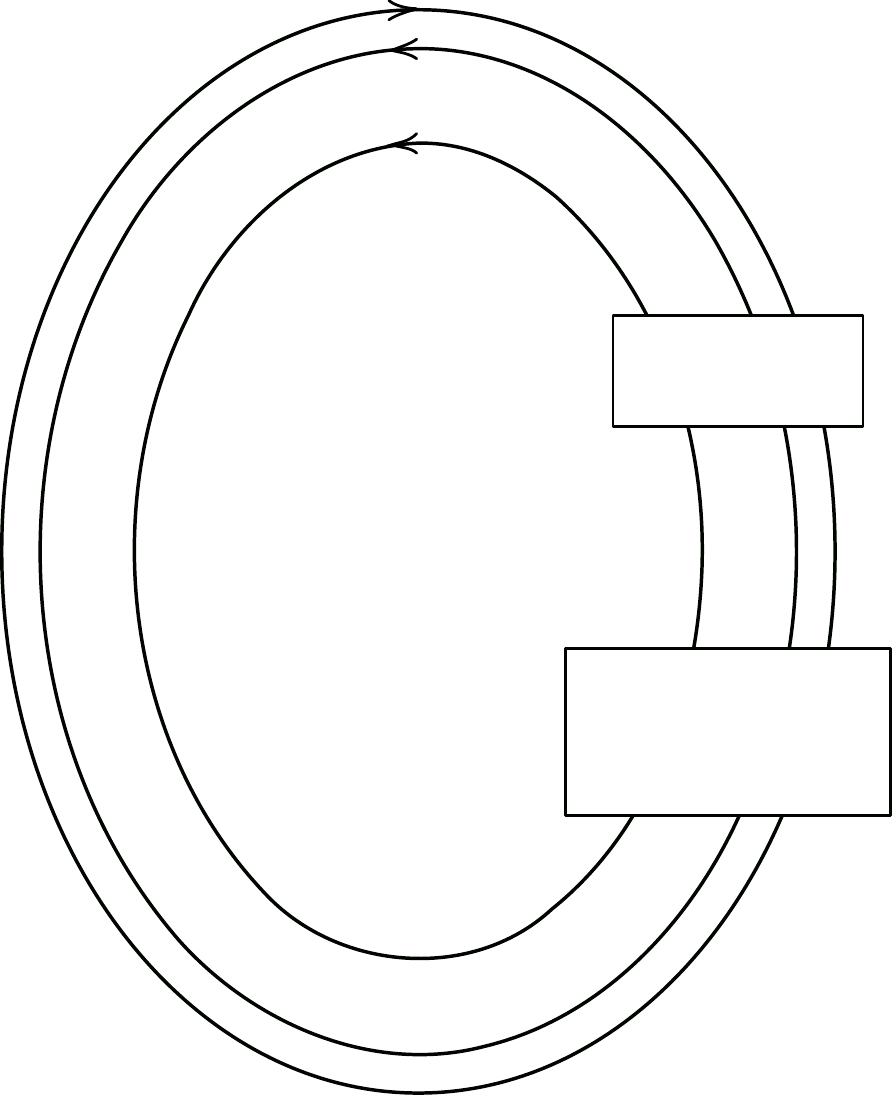}
\put(-0.475,0.75){\Large $K$}
\put(-0.4,1.6){$r$}
\put(-1.88,1.25){\small$\cdots$}
\caption{An $r$--shaking of $K$.}\label{fig:r_shaking}
\end{figure}

\begin{defn}For any knot $K$, $n\geq 0$, and $r\in\mathbb{Z}$, a \textbf{$2n+1$--component $r$--shaking} of $K$ is a collection of $2n+1$ $r$--framed parallels of $K$, where $n+1$ of the parallels are oriented in the direction of $K$ and the $n$ remaining parallels are oriented in the opposite direction.  When the number of components is irrelevant, we use the phrase `$r$--shaking of $K$' for economy. \end{defn}

Figure \ref{fig:r_shaking} gives a schematic representation of an $r$--shaking of $K$. The box containing $r$ indicates that all the strands passing vertically through the box should be given $r$ full twists. The box containing $K$ indicates that all the strands passing vertically through the box should be tied into 0--framed parallels of the tangle corresponding to $K$ (the strands passing upwards have the knot type of $K$ and the strands passing downwards have the knot type of the reverse of $K$.)

\begin{defn}(Alternative Definition) $K$ is \textbf{$r$--shake slice} if some $r$--shaking of $K$ bounds a smooth, properly embedded, compact, connected genus zero surface in $B^4$. 

The knots $K_0$ and $K_1$ are \textbf{$r$--shake concordant} if  there is a smooth, properly embedded, compact, connected, genus zero surface $F$ in $S^3\times [0,1]$, such that $F\,\cap\,S^3 \times \{0\}$ is an $r$--shaking of $K_0$ and $F\,\cap\,S^3 \times \{1\}$ is an $r$--shaking of $K_1$ (although after taking into account the usual orientation conventions the latter will be a $(-r)$--shaking of $-K_1$). $F$ is said to be an \textbf{$r$--shake concordance} between $K_0$ and $K_1$. $K_0$ is said to be \textbf{$(p,\,q)$ $r$--shake concordant }to $K_1$ for $p$, $q \geq 1$ if there is an $r$--shake concordance between them whose boundary consists of a $p$--component $r$--shaking of $K_0$ and a $q$--component $r$--shaking of $K_1$. 

The \textbf{$r$--shake genus} of $K$, denoted $\gsh^r(K)$, is the least genus of a smooth, properly embedded, compact, connected genus zero surface bounded by an $r$--shaking of $K$ in $B^4$.
\end{defn}

\begin{proof}[Proof of the equivalence of the definitions] Suppose the $2$--sphere $S\hookrightarrow W_K^r$ represents the negative of the preferred generator of $H_2(W_K^r)$. By this we mean that, after isotopy, $S$ intersects the added $2$--handle in $2n+1$ parallels of the core, for some $n$, where $n+1$ of these disks are oriented so that their boundaries are $r$--framed parallels of the reverse of $K$ and the others are oriented so that their boundaries are $r$--framed parallels of $K$.  Let $F$ be the oriented genus zero surface obtained from $S$ by deleting the interiors of these disks. Since the induced orientation on the boundary circles is opposite for $F$ compared to that by the recently removed disks, the oriented boundary of $F$ is the desired $r$--shaking of $K$. The converse is proved by reversing these steps.

The proof in the case of shake concordance and shake genus is similar.\end{proof}

\begin{rem}\label{rem:parallelcopies} Since an $r$--shake concordance $F\hookrightarrow S^3\times [0,1]$ has a trivial normal bundle, we can take ``parallel'' copies of it. There are $\pi_1(SO(2))\cong\mathbb{Z}$ trivializations of this bundle and hence an infinite number of choices for a parallel copy. The normal vector field given by the $r$--framing on $\partial F\hookrightarrow S^3\times \{0,1\}$ can be extended to all of $F$ (by linking number considerations). \textit{This is the notion of parallel copy we will always use in this paper.} We will normally want to take $2\ell+1$ parallel copies, $\ell$ of which have altered orientations, which we refer to as an algebraically one number of copies. The reader can easily verify that this notion of parallelism has the following feature: an (algebraically one) number of parallel copies of an $r$--shaking of $K$ is another $r$--shaking of $K$;  and an (algebraically one)  number of parallel copies of $F$ is (after connecting components) another $r$--shake concordance.\end{rem}

Clearly $K_0$ is $(p,\,q)$ $r$--shake concordant to $K_1$ if and only if $K_1$ is $(q,\,p)$ $r$--shake concordant to $K_0$. Thus the relation of $r$--shake concordance is reflexive and symmetric, but \textit{not necessarily transitive}.  However, the following is easily seen to hold.

\begin{prop}\label{prop:easytransitivity} If $K_0$ is $(p,1)$ $r$--shake concordant to $K_1$, and $K_1$ is $(m,1)$ $r$--shake concordant to $K_2$, then $K_0$ is $(pm,1)$ $r$--shake concordant to $K_2$. By symmetry, if $K_0$ is $(1,p)$ $r$--shake concordant to $K_1$, and $K_1$ is $(1,m)$ $r$--shake concordant to $K_2$, then $K_0$ is $(1,\,pm)$ $r$--shake concordant to $K_2$
\end{prop}
\begin{proof} Using Remark~\ref{rem:parallelcopies}, glue $m$ (algebraically one) parallel copies of the $(p,1)$ $r$--shake concordance $F_{01}\hookrightarrow S^3\times [0,1]$ to one copy of the $(m,1)$ $r$--shake concordance $F_{12}\hookrightarrow S^3\times [1,2]$, and observe that it has genus zero and the appropriate boundary. 
\end{proof}

\begin{prop}\label{prop:shakesliceandconcordant} A knot $K$ is $r$--shake slice if and only if $K$ is $(m,1)$ $r$--shake concordant to the unknot, for some $m$. \end{prop}
\begin{proof}In the forward direction, we cut out a small neighborhood in $B^4$ of a point in the genus zero surface bounded by an $r$--shaking of $K$. In the backward direction, we cap off the unknot by its standard slice disk.\end{proof}

We also easily see that for any knot $K$ and integer $r$, $\gsh^r(K)=\gsh^{-r}(-K)$. 
%===========================================================
\section{Characterizing shake concordance of knots}\label{sec:characterization}

In this section we characterize $r$--shake concordance in terms of concordance and certain winding number one satellite operations. 

For simple winding number one patterns $P$, it is sometimes easy to exhibit a genus zero surface cobounded by $P(K)$ and a 0--shaking of $K$, thereby demonstrating that $P(K)$ is $0$--shake concordant to $K$.  One such case is shown in Figure \ref{fig:PKasfusion}. The figure on the left shows a 3--component 0--shaking of $K$ and the figure on the right shows how we may add two bands to obtain the knot $P(K)$, for the Mazur pattern $P$ shown in Figure \ref{fig:P} (for the sake of clarity the attached bands are drawn in a slighter lower weight).

\begin{figure}[t]
\centering
\includegraphics[width=4.5in]{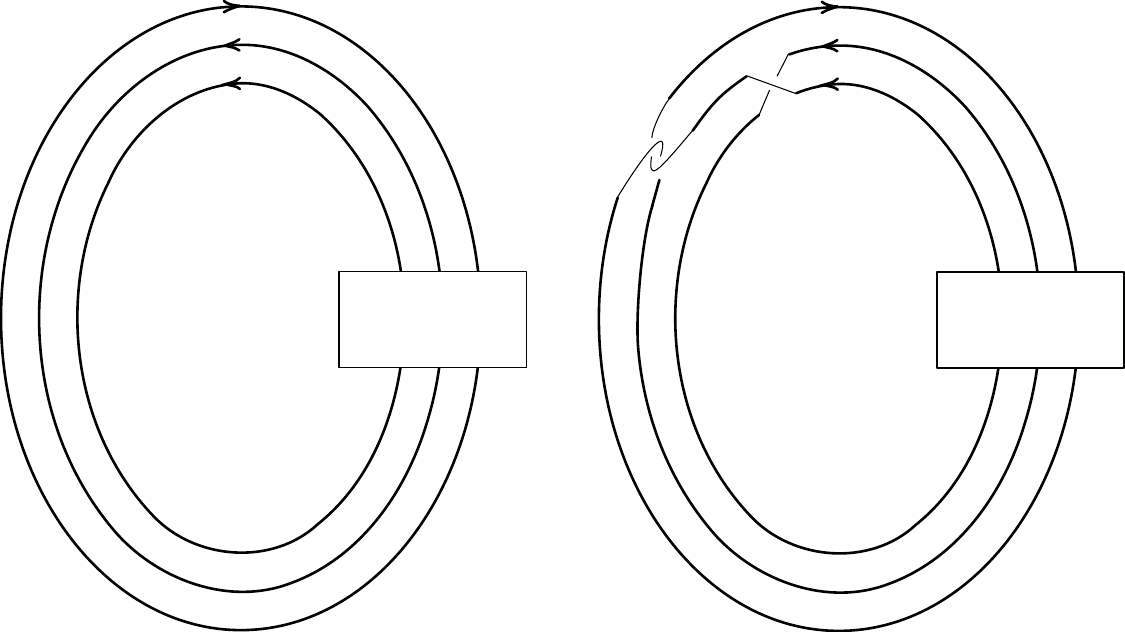}
\put(-0.45,1.175){\Large $K$}
\put(-2.85,1.175){\Large $K$}
\caption{The knot $P(K)$ (right) as a fusion of the 3--component 0--shaking of $K$ (left). The fusion bands are drawn in slightly lower weight for clarity.}\label{fig:PKasfusion}
\end{figure}
This philosophy leads to the following general result.

\begin{prop}\label{prop:satelliteimpliesshake1n} Suppose $P$ is a winding number one pattern  where $\widetilde{P}$ is a slice knot. Then $P_r(K)$ is $(1,\,n)$ $r$--shake concordant to $K$. Moroever, $n\geq 1$ can be taken to be the geometric winding number of $P$. \end{prop}

\begin{proof}  Let $\eta$ denote the meridian of the solid torus $ST$ containing $P$, i.e.\ $ST=S^3-N(\eta)$. From the definition of the satellite construction, it will suffice to show that, within $ST\times [0,1]$, $P\subseteq ST\times\{0\}$ cobounds a genus 0 surface with a 0--shaking of the core of $ST\times\{1\}$, i.e.\ $2k+1$ copies of the core, where $k+1$ copies are oriented in the direction of the longitude of $ST$ and the $k$ remaining copies are oriented in the opposite direction, for some $k\geq 0$. 

Let $\Delta\subseteq S^3\times [0,1]$ be a slice disk for $\widetilde{P}$.  Consider $A=\eta\times [0,1] \subseteq S^3\times [0,1]$. We can assume that $\Delta$ intersects $A$ transversely.  In fact, if the geometric winding number of $P$ is $n$ then we may assume that there are precisely $n$ such points of intersection. Let $x$ be one  intersection point between $A$ and $\Delta$, and $N(x)$ a small ball centered at $x$. The disk $\Delta$ intersects $N(x)$ in a disk and intersects $\partial N(x)$ in a circle, in fact, a meridional circle to $A$. Let $\overline{\Delta}$ be $\Delta - N(x)$. Choose an arc on $A$ connecting $x$ to some point on $\eta\times \{1\}$ in $S^3\times \{1\}$. The restriction to this arc of the unit normal bundle to $A$ is a tube connecting a component of the boundary of $\overline{\Delta}$ to a meridian of $\eta\times \{1\}$ which is a longitudinal circle of  $ST=S^3-N(\eta)$. The latter circle is oriented along the longitude of $ST$ if the intersection at $x$ is positive, and oriented in the opposite direction if the intersection at $x$ is negative. Do this for each point of intersection; the arcs from the intersection points to $S^3\times\{1\}$ can be assumed to be disjoint. By gluing these tubes to $\overline{\Delta}$ we get a genus zero surface $\Sigma$. Notice that $\Sigma \subseteq (S^3-N(\eta))\times [0,1] = ST\times [0,1]$, and is cobounded by $P\subseteq S^3\times\{0\}$ and $n$ copies of the core  of $ST\subseteq S^3\times [0,1]$. Since the algebraic intersection number of $\Delta$ and $A$ is 1, $\Sigma \cap ST\times\{1\}$ is exactly a $0$--shaking of the core of $ST\times\{1\}$. \end{proof}

\begin{cor}\label{cor:satelliteimpliesshake} Suppose for knots $K$ and $J$ there exist winding number one patterns $P$ and $Q$, with $\widetilde{P}$ and $\widetilde{Q}$ slice knots, such that $P_r(K)$ is concordant to $Q_r(J)$ for some $r$. Then $K$ is  $r$--shake concordant to $J$. \end{cor} 
\begin{proof} We are given a concordance $C$ between $P_r(K)$ and $Q_r(J)$. By Proposition~\ref{prop:satelliteimpliesshake1n}, we have an $(m,\,1)$ $r$--shake concordance from  $K$ to $P_r(K)$ (call it $S_1$) and a $(1,\,n)$ $r$--shake concordance from $Q_r(J)$ to $J$ (call it $S_2$), for some $m$, $n \geq 1$. By gluing together $S_1$, $C$, and $S_2$, as in Proposition~\ref{prop:easytransitivity}, we get an $(m,\,n)$ $r$--shake concordance from $K$ to $J$. \end{proof}

In fact, it is sufficient for $P_r(K)$ and $Q_r(J)$ to be merely $r$--shake concordant, as we see in the proposition below. 

\begin{cor}\label{cor:bettersatelliteimpliesshake} For winding number one patterns $P$ and $Q$ with $\widetilde{P}$ and $\widetilde{Q}$ slice knots, and knots $K$ and $J$, if $P_r(K)$ is $r$--shake concordant to $Q_r(J)$, $K$ is $r$--shake concordant to $J$.\end{cor}

\begin{proof}  We are given an $r$--shake concordance $S$ between $P_r(K)$ and $Q_r(J)$ (suppose the boundary consists of a $2k+1$ component $r$--shaking of $P_r(K)$ and a $2l+1$ component $r$--shaking of $Q_r(J)$). By Proposition \ref{prop:satelliteimpliesshake1n}  there exists $S_1$, a $(1,\,m)$ $r$--shake concordance from $P_r(K)$ to $K$, and $S_2$, a $(1,\,n)$ $r$--shake concordance from $Q_r(J)$ to $J$. Using Remark~\ref{rem:parallelcopies}, by gluing,  onto $S$, $2k+1$ copies of $S_1$ (algebraically one) and $2l+1$ copies of $S_2$ (algebraically one), we get a $(m(2k+1),\,n(2l+1))$ $r$--shake concordance from $K$ to $J$. \end{proof} 

If we let $P=Q$ in the above proposition we see that if $P_r(K)$ is $r$--shake concordant to $P_r(J)$, then $K$ is $r$--shake concordant to $J$. This is quite similar to the injectivity result for (untwisted) winding number one satellite operators in the realm of concordance, proved in \cite{CDR14}. In fact, if $\approx_r$ denotes the equivalence relation generated by $r$-shake concordance, we obtain the following corollary.

\begin{cor}\label{cor:shakeinjectivity} If $P$ is a winding number one pattern with $\widetilde{P}$ a slice knot, then the satellite operator $P_r:\C\to \C$ induces a bijective map
$$
P_r:\frac{\mathcal{C}}{\approx_r}\to \frac{\mathcal{C}}{\approx_r},
$$
which is, in fact, the identity map.
\end{cor}
\begin{proof} Suppose $K\approx_r J$. By Proposition~\ref{prop:satelliteimpliesshake1n}, $P_r(K)\approx_r K$ and $P_r(J)\approx_r J$. Thus $P_r$ is well-defined and is the identity function.
\end{proof}

This is rather interesting since A.\ Levine has shown that the Mazur satellite operator, $P:\C\to \C$, whose pattern is shown in Figure~\ref{fig:P}, is far from surjective ~\cite{Lev14}.

\begin{prop}\label{prop:1mshakeimpliessatellite} Suppose $J$ is $(1,m)$ $r$--shake concordant to $K$ for some $m\geq 1$. Then $J$ is concordant to $P_r(K)$ for some winding number one pattern $P$ where $\widetilde{P}$ is a ribbon knot.\end{prop}

\begin{figure}[t]
\centering
\includegraphics[width=5in]{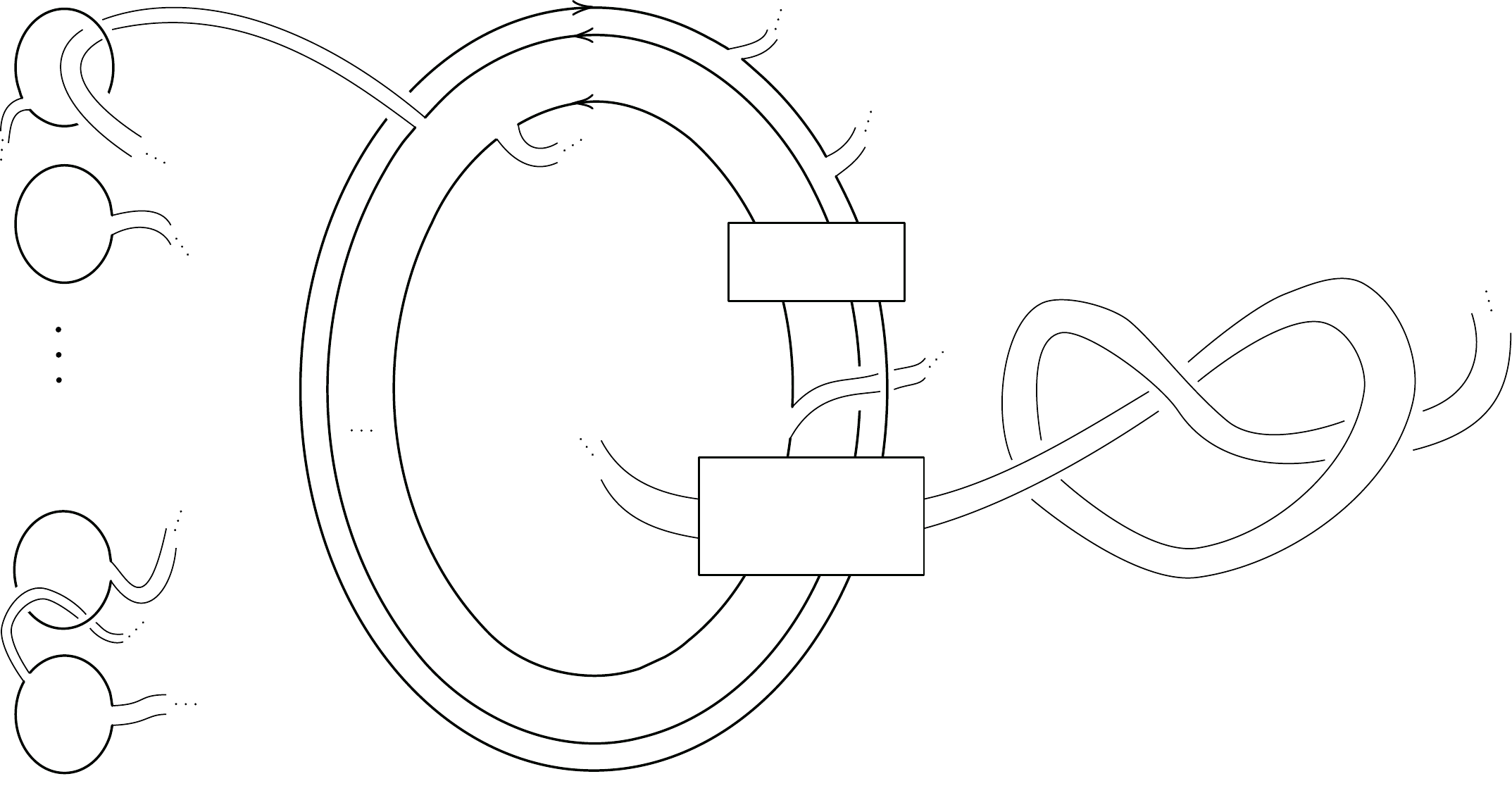}
\put(-2.425,0.85){\Large $K$}
\put(-2.35,1.725){$r$}
\caption{The knot $J'$ shown as a fusion of an $r$--shaking of $K$ and a trivial link $T$ (shown on the left hand side of the picture). The fusions bands are drawn in a lower weight for clarity.}\label{fig:fusion1}
\end{figure}

\begin{proof} Let $F$ be the genus zero surface in $S^3\times [0,1]$ whose boundary is $J\hookrightarrow S^3\times \{1\}$ and an $r$--shaking of $K\hookrightarrow S^3\times \{0\}$. After isotoping $F$ we can assume that the projection map $S^3\times [0,1]\rightarrow [0,1]$ is a Morse function when restricted to $F$ such that all the local maxima occur at level $\{4/5\}$, the split saddles at level $\{3/5\}$, the join saddles at level $\{2/5\}$, and the local minima at level $\{1/5\}$. As a result, the level $\{1/2\}$ is connected, i.e., equals some knot $J'\hookrightarrow S^3\times \{1/2\}$. Hence $J$ is concordant to $J'$. In addition, $J'$ is a fusion of the disjoint union of the $r$--shaking of $K$ (let $m$ be the number of components of the $r$--shaking)  and a trivial link $T$ corresponding to the local minima of $F$. Recall that a fusion of a link $L$ is a link obtained from $L$ by attaching bands that always decrease the number of components. See Figure \ref{fig:fusion1} for a schematic picture of $J'$. Notice that by an isotopy we can ensure that the fusion bands miss the $r$ full twists, i.e.\ the fusion bands do not interact with the box containing $r$. The fusion bands entering the box containing $K$ from the sides indicate that the bands interact with the strands tied into the knot $K$, but need not be tied into the knot $K$ themselves. To complete the proof, it only remains to show that $J'$ is concordant to some $P_r(K)$ as claimed.

\begin{figure}[t]
\centering
\includegraphics[width=5.5in]{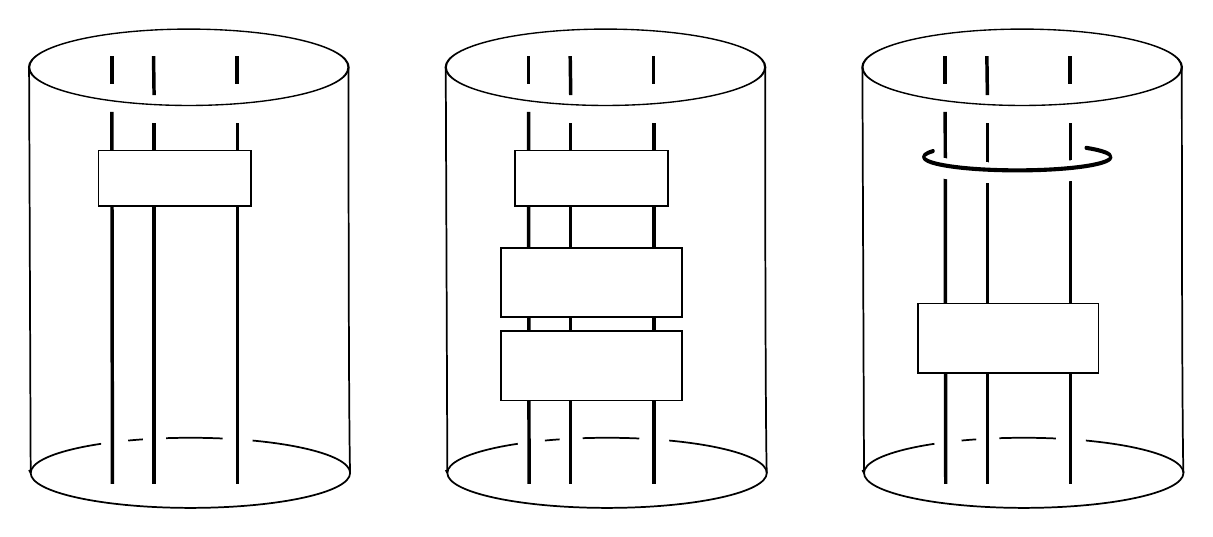}
\put(-5.2,1){$S$}
\put(-4.725,1.6){$r$}
\put(-4.7,1){$\cdots$}
\put(-4.725,-0.1){$L$}
\put(-2.85,1.6){$r$}
\put(-2.85,1.375){$\cdots$}
\put(-2.875,1.1){$K$}
\put(-2.95,0.725){$-K$}
\put(-2.85,-0.1){$L'$}
\put(-0.4,1.7){$\eta$}
\put(-0.95,1.375){$\cdots$}
\put(-1.1,0.85){$-K$}
\put(-0.95,-0.1){$L''$}
\caption{}\label{fig:stringlinkconcordance}
\end{figure}

Let $L$ denote the $m$--component $r$--shaking of $K\hookrightarrow S^3\times\{0\}$. Choose an embedded $B=D^2\times [0,1]$ which intersects $L$ in $m$ trivial strands with $r$ full twists (call this string link $S$---as shown on the left-most image in Figure~\ref{fig:stringlinkconcordance}) and is disjoint from $T$ and the fusion bands; we can do so easily since the fusion bands do not interact with the box containing $r$ in Figure \ref{fig:fusion1}. The string link $S$ is concordant, as a string link, to $m$ $r$--framed parallel copies of an arc in $B$ whose closure has the knot type of the $m$--component $r$--shaking of the slice knot $-K\# K$ (see the center image in Figure~\ref{fig:stringlinkconcordance}). Thus $L$ is concordant to $L'$, the $m$--component $r$--shaking of the knot $K\#-K\#K$. Since the fusion bands are exterior to $B$, it follows that $J'$ is concordant to $J''$, which is a fusion of $L'$ and the trivial link $T$ using the same fusion bands as in $J$. We show below that $J''$ is isotopic to some $P_r(K)$, which will complete the proof.

Let $L''$ be the $m$--component 0--shaking of $K\#-K$. We think of $L''$ as obtained from $L$  by replacing the parallels of $L$ within $B$ by a string link corresponding to $-K$, and removing the $r$ full twists (see the right-most image in Figure~\ref{fig:stringlinkconcordance}). Since $L''$ is a 0--shaking of a ribbon knot, it is a ribbon link. Let $\eta=\partial D^2\times \{1\}$ within $B=D^2\times[0,\,1]$. Then the exterior of $\eta$ in $S^3$ is an unknotted solid torus $ST$ containing $L''$, $T$, as well as the fusion bands. Let $P$ denote the knot in $ST$ obtained as this fusion of $L''$ and $T$. This is a pattern of winding number one. Then note that $J''=P_r(K)$. Moreover, since $\widetilde{P}$ is a fusion of the ribbon link $L''\cup T$, it is a ribbon knot. \end{proof}

\begin{cor}\label{cor:shakeimpliessatellite}If a knot $J$ is $r$--shake concordant to a knot $K$, then there exist winding number one patterns $P$ and $Q$, with $\widetilde{P}$ and $\widetilde{Q}$ ribbon, such that $P_r(K)$ is concordant to $Q_r(J)$ \end{cor}

\begin{proof}Suppose $K$ is $r$--shake concordant to $J$ via an $(m,n)$ shake concordance $F$. By isotoping $F$ we can assume that the projection map $S^3\times [0,1]\rightarrow [0,1]$ is a Morse function when restricted to $F$ such that $F\cap\{1/2\}$ is connected. Call this knot $K'$. We see then that $K'$ is $(1,m)$ $r$--shake oncordant to $K$ and $(1,n)$ $r$--shake concordant to $J$. The proof is completed by applying the preceding proposition. \end{proof}

\begin{thm}\label{thm:main} Two knots $K$ and $J$ are $r$--shake concordant if and only if there exist winding number one patterns $P$ and $Q$, with $\widetilde{P}$ and $\widetilde{Q}$ ribbon, such that $P_r(K)$ is concordant to $Q_r(J)$. \end{thm}

\begin{proof}This follows directly from Corollaries \ref{cor:satelliteimpliesshake} and \ref{cor:shakeimpliessatellite}.\end{proof}

Recall that $\approx_r$ denotes the equivalence relation generated by $r$--shake concordance of knots. 

\begin{cor}\label{cor:sameequivrelation}  For any integer $r$, the equivalence relation on the set of isotopy classes of knots generated by $r$--shake concordance is the same as that generated by concordance together with the relation $K\sim P_r(K)$ for all $K$ and all winding number one patterns $P$ with $\widetilde{P}$ a ribbon knot.
\end{cor}

\begin{proof} Suppose $K \approx_r  J$. Then, since $r$--shake concordance is reflexive and symmetric, there is a sequence of knot types $K=K_0, K_1, \dots , K_n=J$ such that, for each $i$,  $K_i$ is $r$--shake concordant to $K_{i+1}$. Then by the forward direction of Theorem~\ref{thm:main}, for each $i$, there exist winding number one patterns $P^{(i)}$ and $Q^{(i)}$, with $\widetilde{P}^{(i)}$ and $\widetilde{Q}^{(i)}$ ribbon, such that $P_r^{(i)}(K_i)$ is concordant to $Q_r^{(i)}(K_{i+1})$. Now consider the sequence of knot types: $K_0, P_r^{(0)}(K_0), Q^{(0)}_r(K_1), K_1, P_r^{(1)}(K_1),Q^{(1)}_r(K_2),\dots ,Q^{(n-1)}_r(K_n),K_n$. In this sequence, for each $j$, one of three things holds: the $j^{th}$ knot is concordant to the $(j+1)^{th}$ knot,  the $j^{th}$ knot is a winding number one satellite of the $(j+1)^{th}$ knot, or the $(j+1)^{th}$ knot is a winding number one satellite of the $j^{th}$ knot (with $\widetilde{P}$ a ribbon knot for all). This implies that $K$ is related to $J$ in the equivalence relation generated by concordance together with the relation $K\sim P_r(K)$ as stated.

The converse is proved by essentially reversing this argument, using the fact that concordant knots are $r$--shake concordant and the backward direction of Theorem~\ref{thm:main}.
\end{proof}

\begin{cor} \label{cor:characterizeshakeslice} A knot $K$ is $r$--shake slice if and only if there exists a winding number one pattern $P$, with $\widetilde{P}$ ribbon, such that $P_r(K)$ is slice.\end{cor}
\begin{proof}Recall from Proposition \ref{prop:shakesliceandconcordant} that a knot $K$ is $r$--shake slice if and only if it is $(m,\,1)$ $r$--shake concordant to the unknot for some $m$. But by Propositions \ref{prop:satelliteimpliesshake1n} and \ref{prop:1mshakeimpliessatellite}, $K$ being $(m,\,1)$ $r$--shake concordant to the unknot is equivalent to there existing a winding number one pattern $P$ with $\widetilde{P}$ ribbon such that $P_r(K)$ is slice.\end{proof}

In particular, this means that a knot is 0--shake slice if and only if there exists a winding number one pattern $P$, with $\widetilde{P}$ ribbon, such that $P(K)$ is slice. 
\begin{cor}\label{cor:characterizezero2} There exists a $0$--shake slice knot that is not a slice knot if and only if there exists some winding number one satellite operator $P:\C\to\C$, with $\widetilde{P}$ ribbon, which fails to be weakly injective, i.e. there exists a knot $K\neq 0$ (not slice) such that $P(K)=0$ (is slice).\end{cor}

%===========================================================
\section{Shake concordant knots that are not concordant}\label{sec:examples}

The primary goal in this section is to prove Theorem \ref{thm:examples}, which we state below. The proof is postponed until the end of this section. 

\begin{thm}\label{thm:examples}For any integer $r$, there exist infinitely many knots which are distinct in smooth concordance but are pairwise $r$--shake concordant. For $r=0$, there exist topologically slice knots with this property as well. 

In addition, none of $\tau$, $\s$, or slice genus is invariant under $r$--shake concordance, for any integer $r$.\end{thm}

We will use several tools from Legendrian knot theory; see~\cite{Etn05,GomStip99} for excellent introductions to this field. 

\begin{defn}\label{defn:rsuitable} For any integer $r$, a knot $K$ is said to be $r$--suitable if it has some Legendrian representative $\mathcal{K}$ such that 
$$\tb(\mathcal{K})= r \text{ and } \rot(\mathcal{K}) =2\g_4(K)-1-r.$$
\end{defn}

\begin{rem}\label{rem:rsuitableprops1}Any $r$--suitable knot is also $k$--suitable for any $k\leq r$, since positive stabilization of a Legendrian knot decreases Thurston--Bennequin number by one, increases rotation number by one, and preserves topological knot type.  

One can easily transform a diagram of a knot as a positive braid closure to a Legendrian front diagram for a Legendrian representative. The Bennequin inequality~\cite{Eli92}\cite[p.\ 19]{Etn05} then implies that any knot $K$ obtained as the closure of a positive braid has a Legendrian representative $\mathcal{K}$ with $\tb(\mathcal{K})=2\g(K)-1$ and $\rot(\mathcal{K})=0$. Moreover, by the slice--Bennequin inequality \cite{Rud95}\cite[p.\ 133]{Etn05}, $\tb(\mathcal{K})+\lvert \rot(\mathcal{K})\rvert \leq 2\g_4(K)-1$, and therefore, $\g(K)=\g_4(K)$. Therefore, any positive braid closure is $(2\g_4(K)-1)$--suitable. For example, the right-handed trefoil RHT, is 1--suitable and the unknot is $(-1)$--suitable. In fact, any positive knot is $(2\g_4(K)-1)$--suitable with $g(K)=g_4(K)$ by~\cite[Proposition~3.2]{LidSiv14}.\end{rem} 

\begin{rem}\label{rem:suitableequalities} Recall that we have `expanded' versions of the slice--Bennequin inequality, as follows. For any Legendrian representative $\mathcal{K}$ of a knot $K$,
\begin{equation}
\begin{aligned}\label{eqn:slicebennequinexpansions1}
\tb(\K)+\lvert\rot(\K)\rvert\leq 2\tau(K)&-1\leq 2\g_4(K)-1\\
\tb(\K)+\lvert\rot(\K)\rvert\leq \s(K)&-1\leq 2\g_4(K)-1.
\end{aligned}
\end{equation}
 using \cite[Corollary 1.1]{KronMrow13}\cite[Theorem 1.1]{OzSz03}\cite{Plam04, Shuma07}. By examining the above, we see that if a knot $K$ is $r$--suitable, 
\begin{equation}\label{eqn:tauequalsgenus}
2\tau(K)=\s(K)=2\g_4(K),
\end{equation} 
and
\begin{equation}\label{eqn:rnotbig}
r\leq 2\g_4(K)-1.
\end{equation} 
In particular, note that if $r\geq 0$, $K$ cannot be slice. \end{rem}

\begin{lem}\label{lem:whiteheadsuitable} If $K$ is an $r$--suitable knot with $r\geq 0$, then the untwisted positive Whitehead double of $K$, denoted $\Wh(K)$, is 1--suitable.\end{lem}

\begin{proof} Note that $\Wh(K)$, has a Legendrian diagram $\mathcal{W}h(\mathcal{K})$, with $\tb(\mathcal{W}h(\mathcal{K}))=1$ and $\rot(\mathcal{W}h(\mathcal{K}))=0$~\cite{Rud95}\cite[Figure 9]{AkMat97}. Since $K$ is $r$--suitable and non-slice (since $r\geq 0$), $\tau(K)>0$ by~\eqref{eqn:tauequalsgenus}. Then by \cite[Theorem 1.4]{He07}, $\tau(\Wh(K))=1$ and therefore, $\Wh(K)$ is not slice. Since $\g(\Wh(K))=1$, this implies that $\g_4(\Wh(K))=1$ and therefore, $\Wh(K)$ is 1--suitable. \end{proof}

\begin{lem}\label{lem:sumofsuitable} If $K$ is $r$--suitable and $J$ is $k$--suitable, $K\# J$ is $(r+k+1)$--suitable.
\end{lem}
\begin{proof} If $\mathcal{K}$ and $\mathcal{J}$ are Legendrian representatives of the knots $K$ and $J$ respectively, then, $\mathcal{K}\#\mathcal{J}$ is a Legendrian representative of $K\# J$, for which, by ~\cite[p.\ 39]{Etn05}, and Definition~\ref{defn:rsuitable},
$$\tb(\mathcal{K}\#\mathcal{J})=\tb(\mathcal{K})+\tb(\mathcal{J})+1=r+k+1,$$
 and 
$$\rot(\mathcal{K}\#\mathcal{J})=\rot(\mathcal{K})+\rot(\mathcal{J})=2\left(\g_4(K)+\g_4(J)\right)-1-(r+k+1).$$
Thus it suffices to show that $\g_4(K\#J)=\g_4(K)+\g_4(J)$. Clearly $\g_4(K\#J)\leq \g_4(K)+\g_4(J)$ holds for all knots. Conversely,
$$\g_4(K\#J)\geq \tau(K\#J)=\tau(K)+\tau(J)=\g_4(K)+\g_4(J),$$
by ~(\ref{eqn:slicebennequinexpansions1}) and ~(\ref{eqn:tauequalsgenus}).
\end{proof}

\begin{rem}\label{rem:rsuitableprops2} Since $\Wh(RHT)$ is 1--suitable by Lemma~\ref{lem:whiteheadsuitable} (and therefore, $k$--suitable for any $k\leq 1$), for any fixed integer $r$, we can find a topologically slice knot $K$ that is $r$--suitable by letting $K$ be the connected sum of $\max \{1,\lceil \frac{r+1}{2} \rceil\}$ copies of $\Wh(RHT)$, by the above lemma. \end{rem}

\begin{defn}\label{defn:4genusofpattern}Let $P$ be a winding number one pattern, i.e.\ a knot inside the unknotted standard solid torus $ST$. The slice genus of $P$, denoted $\g_4(P)$, is the least genus of a surface $\Sigma\subseteq  ST\times [0,1]$ cobounded by $P\hookrightarrow ST\times\{0\}$ and the core of $ST\times \{1\}$. \end{defn}

Note that it follows that $P$ must also cobound a surface of genus $\g_4(P)$ with a $k$--twisted longitude of $ST\times \{1\}$, for any value of $k$. 

\begin{rem}\label{rem:4genusofmazur}Clearly, the slice genus of the trivial pattern (given by the core of the solid torus) is zero. The slice genus of the Mazur pattern is one, as follows. Let $P$ denote the Mazur pattern. We know that $g_4(P)\leq 1$ since, by changing a single crossing,  $P$ can be transformed to a pattern isotopic to the core of $ST$.  On the other hand  it follows that $\g_4(P)\neq 0$, since if $P$ were concordant to the core of $ST$, then $P(K)$ would be concordant to $K$ for any $K$. But in~\cite[Section 3]{CFHeHo13} this was shown not to be the case, in particular for the right-handed trefoil. Hence $\g_4(P)= 1$.  \end{rem}

\begin{prop}\label{prop:distinctiterates}Fix an integer $r$. Let $P$ be any winding number one pattern with a Legendrian diagram $\mathcal{P}$ such that 
$$\tb(\mathcal{P})=0 \text{ and } 0<\g_4(P)\leq \frac{\rot(\mathcal{P})}{2}$$
For any $r$--suitable knot $K$, the iterated satellites $\{P^i_r(K)\mid i\geq 0\}$ are distinct smooth concordance classes. 

Moreover, 
\begin{align*}
\g_4(P^i_r(K))&= \g_4(K)+i\cdot\g_4(P),\\
\tau(P^i_r(K))&= \tau(K)+i\cdot\g_4(P),\\
\s(P^i_r(K))&= \s(K)+2i\cdot\g_4(P). 
\end{align*}
\end{prop}

\begin{rem}In the statement above, the notation $P^i_r(K)$ denotes the iterated satellite knot $P_r(P_r(\cdots (K)\cdots ))$, but in fact this is the same knot as the one obtained by constructing the iterated pattern $P^i$ (see~\cite[Section 2.1]{DR13} or~\cite[Section 2.1]{Ray14}) and then constructing the twisted satellite~$(P^i)_r(K)$. This is shown in Proposition~\ref{prop:iteratesorcompose}, by examining the gluing maps in the two a priori different constructions. \end{rem}

\begin{rem} The expression $\frac{\rot(\mathcal{P})}{2}$ in the statement of Proposition~\ref{prop:distinctiterates} is in fact an integer (this is not used in our proofs). More generally, for a Legendrian diagram $\mathcal{P}$ for a pattern $P$ with winding number $w(P)$, $\tb(\mathcal{P})$ and $\rot(\mathcal{P})$ have the same parity if $w(P)$ is odd, and opposite parities if $w(P)$ is even. This is due to the fact that the Thurston--Bennequin number and the rotation number of a Legendrian knot have different parities~\cite[Proposition 3.5.23]{Gei08}, and the fact that a Legendrian diagram $\mathcal{P}$ for a pattern $P$ can be changed to yield a  Legendrian representative for the knot $\widetilde{P}$ by introducing twice as many cusps as there are strands in $\mathcal{P}$. The result follows since $w(P)$ and the number of strands of $\mathcal{P}$ have the same parity. \end{rem}

Before giving the proof of Proposition~\ref{prop:distinctiterates}, we point out the following corollary. 

\begin{cor}\label{cor:mazurdistinctiterates}Fix an integer $r$ and let $P$ denote the Mazur operator. The iterated satellite knots $\{P^i_r(K)\mid i\geq 0\}$ correspond to distinct smooth concordance classes, and moreover, 
\begin{align*}
\g_4(P^i_r(K))&= \g_4(K)+i,\\
\tau(P^i_r(K))&= \tau(K)+i,\\
\s(P^i_r(K))&= \s(K)+2i.
\end{align*}
\end{cor}

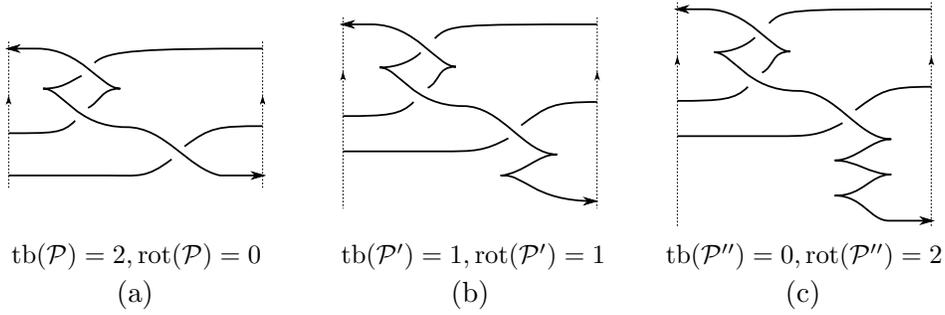
\begin{figure}[t]
\begin{picture}(5,1.75)
\put(0,0.5){\includegraphics[width=1.5in]{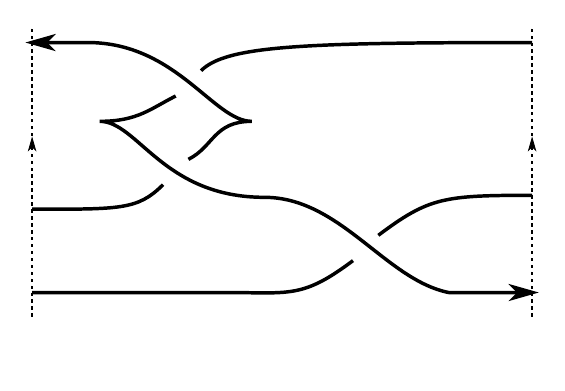}}
\put(0.1,0.3){\small $\tb(\mathcal{P})=2, \rot(\mathcal{P})=0$}
\put(0.65,0.1){(a)}
\put(1.75,0.4){\includegraphics[width=1.5in]{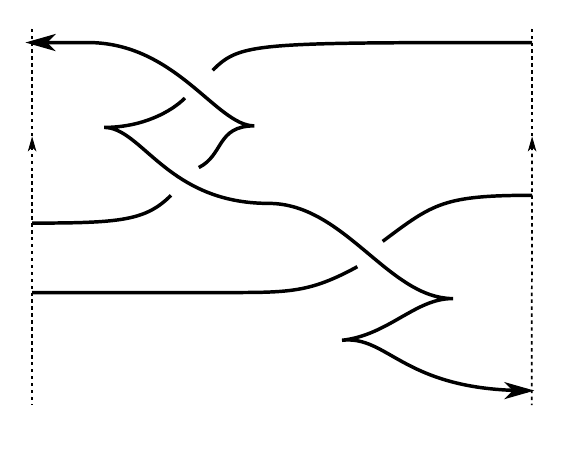}}
\put(1.825,0.3){\small $\tb(\mathcal{P}')=1, \rot(\mathcal{P}')=1$}
\put(2.4,0.1){(b)}
\put(3.5,0.3){\includegraphics[width=1.5in]{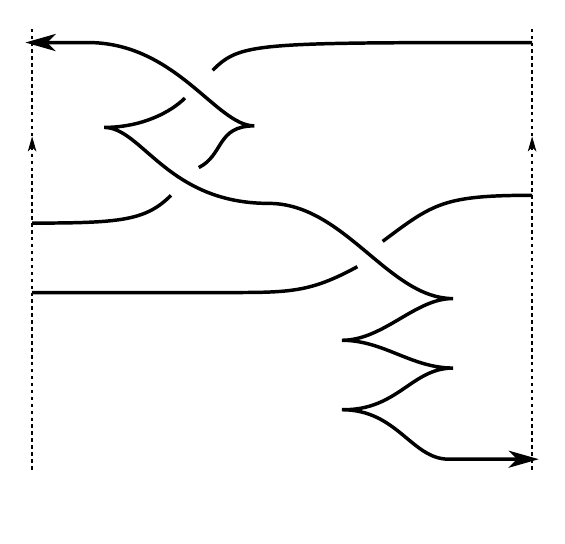}}
\put(3.525,0.3){\small$\tb(\mathcal{P}'')=0, \rot(\mathcal{P}'')=2$}
\put(4.15,0.1){(c)}
\end{picture}
\caption{Legendrian diagrams $\mathcal{P}$ (left), $\mathcal{P}'$ (center), and $\mathcal{P}''$ (right) for the Mazur pattern $P$. Note that $\mathcal{P}'$ and $\mathcal{P}''$ are obtained by from $\mathcal{P}$ and $\mathcal{P}'$ respectively, by performing positive stabilization.}\label{fig:legmazur}
\end{figure}

\begin{proof} Figure \ref{fig:legmazur}(b) shows a Legendrian diagram $\mathcal{P}''$ for the Mazur pattern. We know from Remark~\ref{rem:4genusofmazur} that $\g_4(P)=1$. Thus $\mathcal{P}''$ satisfies the requirements of Proposition \ref{prop:distinctiterates}. The result follows.\end{proof}

\begin{rem} The case $r=0$ for Corollary~\ref{cor:mazurdistinctiterates} follows from the main theorem of \cite{Ray14}; the fact that $P(K)$ and $K$ give distinct concordance classes, and in particular have distinct slice genera, and $\tau$ and $\s$ invariants, was shown earlier in \cite{CFHeHo13}. \end{rem}

\begin{proof}[Proof of Proposition \ref{prop:distinctiterates}]The proof is a variation on the techniques of \cite{Ray14} and \cite[Theorem~3.1]{CFHeHo13}; the primary tool is the slice--Bennequin inequality \cite{Rud95}\cite[p. 133]{Etn05}.

\begin{figure}[t]
\begin{center}
\begin{picture}(5,2.5)
\put(0.25,0.65){\includegraphics[width=1.5in]{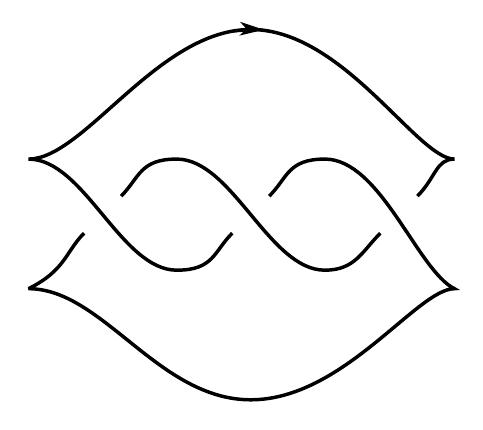}}
\put(0.95,0){$\K$}
\put(2.25,0.15){\includegraphics[width=2.5in]{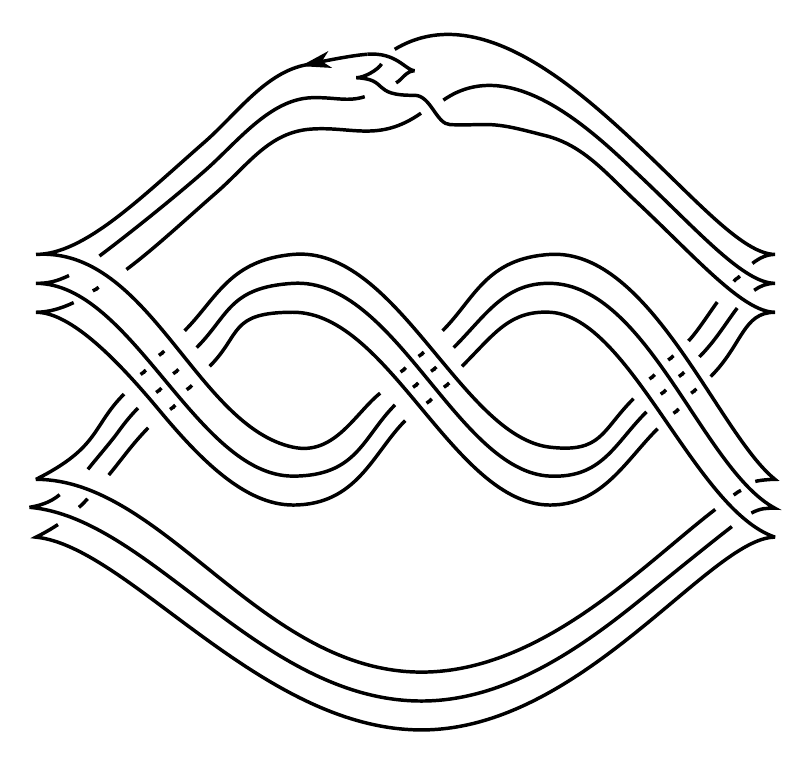}}
\put(3.4,0){$\mathcal{P}(\K)$}
\end{picture}
\end{center}
\caption{The Legendrian satellite operation. Left: a Legendrian representative $\K$ for the right-handed trefoil with $\tb(\K)=1$. Right: the Legendrian satellite $\mathcal{P}(\K)$ where $\mathcal{P}$ is the diagram for the Mazur pattern shown in Figure~\ref{fig:legmazur}(a). The reader may verify that $\mathcal{P}(\K)$ has the topological knot type of the 1--twisted satellite of the right-handed trefoil with companion the Mazur pattern. }\label{fig:legsatelliteoperation}
\end{figure}

Recall that given a Legendrian diagram $\mathcal{P}$ for a pattern $P$ and $\mathcal{K}$ a Legendrian representative of a knot $K$, the Legendrian satellite operation yields $\mathcal{P}(\mathcal{K})$, a Legendrian representative for the $\tb(\mathcal{K})$--twisted satellite of $K$ with pattern $P$ (see~\cite{Ng01} for an overview of the Legendrian satellite operation, and Figure~\ref{fig:legsatelliteoperation} for a picture). Since $K$ is $r$--suitable, we have a Legendrian representative $\mathcal{K}$ for $K$, with 
$$\tb(\mathcal{K})= r \text{ and }\rot(\mathcal{K})=2\g_4(K)-1-r.$$
Let $\mathcal{P}$ be the Legendrian diagram for pattern $P$ given in our hypotheses. Then, $\mathcal{P}(\mathcal{K})$ is a Legendrian representative for $P_r(K)$ since $\tb(\mathcal{K})=r$.  

By ~\cite[Remark 2.4]{Ng01}, since $P$ is winding number one, we see that 
$$\tb(\mathcal{P}(\mathcal{K}))=\tb(\mathcal{P})+\tb(\mathcal{K})=r$$ 
and 
$$\rot(\mathcal{P}(\mathcal{K}))=\rot(\mathcal{P})+\rot(\mathcal{K})=\rot(\mathcal{P})+(2\g_4(K)-1-r).$$  
By hypothesis $\rot(\mathcal{P})\geq 0$ and, by ~\eqref{eqn:rnotbig}, $2\g(K)-1-r\geq 0$. Therefore, 
\begin{align*}
\tb(\mathcal{P}(\mathcal{K}))+\lvert\rot(\mathcal{P}(\mathcal{K}))\rvert & = r+\lvert\rot(\mathcal{P})+(2\g_4(K)-1-r)\rvert \\
&= \rot(\mathcal{P})+2\g_4(K)-1
\end{align*}
But by the slice--Bennequin inequality~(\ref{eqn:slicebennequinexpansions1}), we have that  
$$\tb(\mathcal{P}(\mathcal{K}))+\lvert\rot(\mathcal{P}(\mathcal{K}))\rvert\leq  2\g_4(P_r(K))-1.$$
from which it follows that
\begin{equation}\label{eqn:ref1}
\frac{\rot(\mathcal{P})}{2}+\g_4(K)\leq \g_4(P_r(K)).
\end{equation}
By hypothesis, $\g_4(P)\leq\frac{\rot(\mathcal{P})}{2}$, so 
\begin{equation}\label{eqn:ref2}
\g_4(P)+\g_4(K)\leq \g_4(P_r(K)).
\end{equation}

Since $\g_4(P)>0$ by hypothesis,  we have shown that $\g_4(K)<\g_4(P_r(K))$. In other words, the satellite operator $P_r$ increases the slice genus by at least one when applied to an $r$--suitable knot and thus, the knots $P_r(K)$ and $K$ represent distinct concordance classes. To complete the proof, we would like to iterate this process by applying $P_r$ again, since then we will have shown that the slice genera of the sequence $P^i_r(K)$ is a strictly increasing function of $i$, implying that the $P^i_r(K)$ represent distinct concordance classes. To iterate we need to show that $P_r(K)$ is itself $r$--suitable whenever $K$ is. This is accomplished by the $m=0$ case of the following lemma. 

\begin{lem}\label{lem:iteratesuitable} For any integer $r$, if $K$ is an $(r+m)$--suitable knot for some $m\geq 0$, and $P$ is a pattern with a Legendrian diagram $\mathcal{P}$ such that 
$$\tb(\mathcal{P})=m \text{ and } 0<\g_4(P)\leq \frac{m+\rot(\mathcal{P})}{2},$$
then $P_r(K)$ is an $(r+m)$--suitable knot, and $\g_4(P_r(K))=\g_4(K)+\g_4(P)$. \end{lem}

\begin{proof} Since $K$ is $(r+m)$--suitable, we have a Legendrian representative $\mathcal{K}'$ for $K$, with 
$$\tb(\mathcal{K}')= r+m \text{ and }\rot(\mathcal{K}')=2\g_4(K)-1-r-m.$$
By positive stabilization we arrive at a Legendrian representative $\mathcal{K}$ for $K$ with
$$\tb(\mathcal{K})= r \text{ and }\rot(\mathcal{K})=2\g_4(K)-1-r.$$
Let $\mathcal{P}$ be the hypothesized Legendrian diagram for the pattern $P$. Then $\mathcal{P}(\mathcal{K})$ is a Legendrian representative for $P_r(K)$ since $\tb(\mathcal{K})=r$.    We have seen that $\tb(\mathcal{P}(\mathcal{K}))=r+m$ and $\rot(\mathcal{P}(\mathcal{K}))=\rot(\mathcal{P})+2\g_4(K)-1-r.$ To show that $P_r(K)$ is $(r+m)$--suitable it now suffices to show that 
$$\g_4(P_r(K))=\g_4(K)+\frac{m+\rot(\mathcal{P})}{2}.$$ 

By hypothesis there is a surface $\Sigma$ of genus $\g_4(P)$ inside $ST\times [0,1]$ cobounded by $P\hookrightarrow ST\times\{0\}$ and the $(-r)$--twisted longitude of the core of $ST\times\{1\}$.  Under the embedding $f_r:ST\hookrightarrow S^3$ that defines the $r$--twisted satellite construction, this twisted longitude is identified with the untwisted longitude of $K$. Therefore the image of $\Sigma$ under $f_r\times id: ST\times [0,1]\hookrightarrow S^3\times [0,1]$ is a surface of genus $\g_4(P)$ cobounded by $P_r(K)\hookrightarrow S^3\times \{0\}$ and $K\hookrightarrow S^3\times \{1\}$.  Thus, 
\begin{equation}\label{eq:A}
\g_4(P_r(K))\leq \g_4(K)+\g_4(P) \leq \g_4(K)+\frac{m+\rot(\mathcal{P})}{2},
\end{equation} 
using our hypothesis on $g_4(P)$. 

On the other hand, by the slice--Bennequin inequality ~(\ref{eqn:slicebennequinexpansions1}), we see that  
$$\tb(\mathcal{P}(\mathcal{K}))+\lvert\rot(\mathcal{P}(\mathcal{K}))\rvert\leq  2\g_4(P_r(K))-1.$$

We know that $2\g_4(K)-1-r\geq m$ by ~(\ref{eqn:rnotbig}), and $\rot(\mathcal{P})\geq -m$ by hypothesis. Therefore, 
\begin{align*}
\tb(\mathcal{P}(\mathcal{K}))+\lvert\rot(\mathcal{P}(\mathcal{K}))\rvert & = r+m+\lvert\rot(\mathcal{P})+(2\g_4(K)-1-r)\rvert \\
&= m+\rot(\mathcal{P})+2\g_4(K)-1
\end{align*}
from which it follows that
\begin{equation}\label{eqn:ref3}
\g_4(K)+\g_4(P)\leq\g_4(K)+\frac{m+\rot(\mathcal{P})}{2}\leq \g_4(P_r(K)).
\end{equation}

By combining ~\eqref{eqn:ref3} and ~\eqref{eq:A},  we see that all $4$ inequalities that appear are in fact equalities, so
$$\g_4(P_r(K))=\g_4(K)+\frac{m+\rot(\mathcal{P})}{2},$$
finishing the proof that $P_r(K)$ is $(r+m)$--suitable; and  $\g_4(P_r(K))=\g_4(K)+\g_4(P)$, completing the proof.\end{proof}

This finishes the proof of the assertion in Proposition~\ref{prop:distinctiterates} that the $\{P_r^i(K)\mid i\geq 0\}$ represent distinct concordance classes since, by Lemma~\ref{lem:iteratesuitable}, each $P^i_r(K)$ is $r$--suitable whenever $K$ is $r$--suitable.

But Lemma~\ref{lem:iteratesuitable} actually establishes something stronger. For, not only does each successive application of $P_r$ increase the slice genus, but it increases the slice genus by precisely $\g_4(P)$. Thus we see that $\g_4(P^i_r(K))=\g_4(K)+i\cdot\g_4(P)$.  Finally, by~\eqref{eqn:tauequalsgenus}, it follows that
\begin{align*}
\g_4(P^i_r(K))&= \g_4(K)+i\cdot\g_4(P),\\
\tau(P^i_r(K))&= \tau(K)+i\cdot\g_4(P),\\
\s(P^i_r(K))&= \s(K)+2i\cdot\g_4(P).
\end{align*}This completes the proof of Proposition~\ref{prop:distinctiterates}.
\end{proof}

\begin{proof}[Proof of Theorem \ref{thm:examples}]  From Remark~\ref{rem:rsuitableprops2} we know that for any $r$, we can find a topologically slice $r$--suitable knot $K$. Let $P$ denote the Mazur pattern. From Corollary \ref{cor:mazurdistinctiterates}, we see that the iterated satellites $\{P_r^i(K)\mid i\geq 0\}$ represent distinct concordance classes, with distinct slice genera, and distinct $\tau$ and $\s$ invariants, for any $r$--suitable knot $K$. In the case $r=0$, if $K$ is topologically slice, then $P(K)$ is topologically concordant to $P(U)$, which is unknotted (recall that $P(K)$ denotes the \textit{untwisted} satellite of $K$ with pattern $P$). By induction, if $K$ is topologically slice, $P^i(K)$ is topologically slice for each $i\geq 0$. 

The proof will be completed if we show that, for any $r$, and any $j>i$, $P^i_r(K)$ is $r$--shake concordant to $P^j_r(K)$.  Since  $P^{i+1}_r(K)=P_r(P^i_r(K))$ by definition, we see that by  Proposition~\ref{prop:satelliteimpliesshake1n}, $P^{i+1}_r(K)$ is $(1,\,n)$ $r$--shake concordant to $P^{i}_r(K)$ where $n$ is the geometric winding number of $P$ (thus, $n=3$). Similarly, $P^{i+1}_r(K)$ is $(1,\,n)$ $r$--shake concordant to $P^{i+2}_r(K)$.  By Proposition~\ref{prop:easytransitivity}, $P^{i+2}_r(K)$ is $(1,n^2)$ $r$--shake concordant to $P^{i}_r(K)$.  By repeating these steps, we conclude that $P^j_r(K)$ is $(1,\, n^{j-i})$ $r$--shake concordant to $P^i_r(K)$ for any $j>i$.  \end{proof}

Note that while we restricted ourselves to the Mazur pattern in the above proof, in light of Propositions~\ref{prop:distinctiterates}~and~\ref{prop:satelliteimpliesshake1n}, we could have used any winding number one pattern $P$ with $\widetilde{P}$ slice and a Legendrian diagram $\mathcal{P}$ such that 
$$\tb(\mathcal{P})=0 \text{ and } 0<\g_4(P)\leq \frac{\rot(\mathcal{P})}{2}.$$

%==========================================================================
\section{Obstructions to shake concordance}\label{sec:obstructions}

In this section we point out some elementary obstructions to shake concordance. 

\begin{prop} \label{prop:shakeconchomcob} If $K$ is $r$--shake concordant to $J$, $M^r_K$ is homology cobordant to $M^r_J$ in such a way that the (positive) meridian of $K$ is homologous to that of $J$. Consequently, if $K$ is 0--shake slice, it bounds a smoothly embedded disk in a homology $B^4$. \end{prop}
\begin{proof} Let $\Sigma\hookrightarrow W^r_{K,J}$ be the embedded $2$--sphere guaranteed by the definition of $r$--shake concordance given in Section \ref{sec:intro}. Note that the two boundary components of $W^r_{K,J}$ are  $M_K^r$ and $-M_J^r$.  Note also that $\Sigma$ has a trivial normal bundle; as a result, we can perform surgery on $\Sigma$, i.e.\ cut out a regular neighborhood of $\Sigma$---which is diffeomorphic to $S^2\times D^2$---and glue in a copy of $D^3\times S^1$. A Mayer--Vietoris argument then shows that the resulting 4--manifold is a homology cobordism between $M_K^0$ and $M_J^0$. 

Recall that $K$ is 0--shake slice if and only if it is $(m,\,1)$ 0--shake concordant to the unknot. Therefore, by the above proof, if $K$ is 0--shake slice, $M_K^0$ is homology cobordant to $M_U^0\cong S^1 \times S^2$. We can cap off $M_U^0$ with a $S^1 \times D^3$. This gives a 4--manifold $V$ which is a homology circle with $\partial V=M_K^0$. It is then well-known that $K$ is slice in a homology 4--ball (see for example ~\cite[Proposition 1.2]{CFHeHo13}). We see this by attaching a 0--framed 2--handle to $V$ along the meridian of $K$; call the resulting 4--manifold $W$. Observe that $\partial W \cong S^3$. A Mayer--Vietoris argument shows that $W$ is a homology ball. Moreover, the co-core of the 2--handle is a disk bounded by $K$ in $W$. \end{proof}

Proposition~\ref{prop:shakeconchomcob} is particularly important since many concordance invariants are determined by a knot's zero surgery manifold. 

\begin{cor}\label{cor:algconc} If $K$ is 0--shake concordant to $J$, then the algebraic concordance class of $K$ is equal to that of $J$. In particular $K$ and $J$ have equal signatures and Arf invariants.
\end{cor}

\begin{cor}\label{cor:tauzero}If $K$ is 0--shake slice, $\tau(K)=0$.\end{cor}
\begin{proof}This follows from Proposition~\ref{prop:shakeconchomcob} and \cite[Theorem 1.1]{OzSz03}.\end{proof}

The following result is known, but since a proof does not appear in print, we provide one.

\begin{cor}[\cite{Rob65}]\label{cor:arfzero}If $K$ is $r$--shake slice, $\Arf(K)=0$ \end{cor}

\begin{proof}In \cite[Theorem 2]{Rob65}, Robertello showed that given two knots $K$ and $J$, if each cobounds a genus zero surface in $S^3 \times [0,1]$ with a `proper' link $L$, then $\Arf(K) = \Arf(J)$. A proper link $\sqcup_{i=1}^n L_i$ is a link where $\Sigma_{j\neq i}\, \lk(L_i,L_j)$ is even for each $i$. Since $K$ is $r$--shake slice, there is some $k$ for which $K$ is $(2k + 1, 1)$ $r$--shake concordant to the unknot. Moreover, $K$ and a $2k + 1$--component $r$--shaking of $K$ also cobound a genus zero surface in $S^3 \times [0, 1]$---we can see this by adding bands between oppositely oriented components of the $r$--shaking to cancel all components but one. It is easy to see that any $r$--shaking of a knot is a proper link.\end{proof}

Even in the case $r\neq 0$ certain Tristram signatures obstruct a knot being $r$--shake slice as pointed out in ~\cite{Ak77}.

%==========================================================================
\section{$r$--shake genus of knots}\label{sec:shakegenus}

In this section we establish what we will call an  $r$--shake slice--Bennequin inequality. Using this we show that the Thurston--Bennequin numbers of Legendrian representatives can obstruct a knot from being $r$--shake slice and more generally can give a lower bound on the $r$--shake genus. We show that the sequence of $r$--shake genera of the previously considered families $\{P^i_r(K)~|~i\geq 0\}$ is increasing, implying that the $r$--shake genus is not invariant under $r$--shake concordance!

\begin{prop}[\cite{AkMat97,LisMat98}]\label{prop:preshakeslicebennequin}Fix a knot $K$. Let $\mathcal{K}$ be a Legendrian representative of $K$ and $\Sigma\subseteq B^4$ a smooth properly embedded surface for which $\partial \Sigma$ is an $r$--shaking of $K$ such that $r\leq \tb(\mathcal{K})-1$. Then
$$\tb(\mathcal{K})+\lvert \rot(\mathcal{K})\rvert \leq 2\g(\Sigma)-1.$$
\end{prop}

This yields the following useful corollaries.

\begin{cor}[$r$--shake slice--Bennequin inequality]\label{cor:shakeslicebennequin}For any Legendrian representative $\mathcal{K}$ of a knot $K$ with $\tb(\mathcal{K})-1\geq r $, $$\tb(\mathcal{K})+\lvert\rot(\mathcal{K})\rvert\leq 2\gsh^r(K)-1.$$\end{cor}

%\begin{cor}\label{cor:othershakeslicebennequin}For any fixed integer $r$ and any knot $K$ which is \textit{not} $r$--shake slice,
%$$\tb(\mathcal{K})+\lvert\rot(\mathcal{K})\rvert\leq 2\gsh^r(K)-1,$$
%for any Legendrian representative $\mathcal{K}$ of $K$ with $\tb(\mathcal{K})\geq r+1$. \end{cor}

Let $\TB(K)$ be the maximal Thurston--Bennequin number of $K$, i.e.\ the maximal Thurston--Bennequin number over all Legendrian representatives of $K$. 

\begin{cor}\label{cor:notrshakeslice} If $\TB(K)\geq 1$ then $K$ is not $r$--shake slice for any $r< \TB(K)$; in particular, $\gsh^r(K)\geq\frac{1}{2}(\TB(K)+1)$ and $K$ is not 0--shake slice.\end{cor}

\begin{proof}[Proof of Proposition \ref{prop:preshakeslicebennequin}] We claim that we may assume that $\rot(\mathcal{K})\geq 0$. Changing the orientation of a Legendrian knot changes the sign of the rotation number, but leaves the Thurston--Bennequin number unchanged---this follows from the combinatorial definition of Thurston--Bennequin number and rotation number. Moreover an $r$--shaking of the topological knot type of $rK$ (the reverse of $K$) bounds the surface $\Sigma$ with reversed orientation. Thus, since $\tb(r\mathcal{K})=\tb(\mathcal{K})$, $\lvert\rot(r\mathcal{K})\rvert=\lvert\rot(\mathcal{K})\rvert$, and $\g(r\Sigma)=\g(\Sigma)$, it suffices to prove the statement for either orientation of $\mathcal{K}$, and so we may as well pick the orientation with non-negative rotation number. 

Positive stabilization \cite[p.\ 15]{Etn05} decreases the Thurston--Bennequin number by one at the expense of increasing the rotation number by one. Repeated positive stabilization yields $\mathcal{K}'$, also a Legendrian representative of $K$, with $\tb(\mathcal{K}')=r+1$ and $\rot(\mathcal{K}')=\rot(\mathcal{K})+\tb(\mathcal{K})-r-1$.  

Attach a 2--handle to $B^4$ along $\mathcal{K}'$ with framing $r$. Since we have matched the framings, we can cap off $\Sigma$ with several copies of the core of the attached 2--handle to get a closed surface $\overline{\Sigma}$ with genus $\g(\overline{\Sigma})=\g(\Sigma)$. Moreover, since the 2--handle was attached along the Legendrian knot $\mathcal{K}'$ with framing $\tb(\mathcal{K}')-1$, the resulting 4--manifold $X$ admits a Stein structure~\cite[Proposition 2.3]{Gom98}. By~\cite{AkMat97,LisMat98} (see~\cite[Theorem 3.4]{AkYas13}),

\begin{equation}\label{eqn:adjunction}
2\g(\overline{\Sigma})-2\geq [\overline{\Sigma}]^2 + \lvert c_1(X)([\overline{\Sigma}])\rvert.
\end{equation}

However, note that $[\overline{\Sigma}]^2=r$ (since this is the framing with which the 2--handle was attached). Moreover, $\rot(\mathcal{K}')=c_1(X)([\overline{\Sigma}])$ by~\cite[Proposition 2.3]{Gom98}\cite[Theorem 11.3.1]{GomStip99}. Therefore, 
$$r+\lvert\rot(\mathcal{K}')\rvert \leq 2g(\Sigma)-2.$$
It follows that 
$$(r+1) +\lvert \rot(\mathcal{K})+\tb(\mathcal{K})-r-1\rvert \leq 2\g(\Sigma)-1.$$
Since $\rot(\mathcal{K})$ and $\tb(\mathcal{K})-r-1$ are both non-negative, we see that $$r+1 + \rot(\mathcal{K}) +\tb(\mathcal{K})-r-1 \leq 2\g(\Sigma)-1,$$ that is, since $\rot(\mathcal{K})\geq 0$, $$\tb(\mathcal{K}) + \lvert \rot(\mathcal{K})\rvert \leq 2\g(\Sigma)-1$$ as needed. \end{proof}

\begin{proof}[Proof of Corollary \ref{cor:shakeslicebennequin}]This follows immediately from Proposition \ref{prop:preshakeslicebennequin} and the definition of $r$--shake genus. \end{proof}

%\begin{proof}[Proof of Corollary \ref{cor:othershakeslicebennequin}]Fix an integer $r$. Let $\Sigma$ be a smooth surface bounded by some $r$--shaking of $K$. Since $K$ is not $r$--shake slice, $\g(\Sigma)>0$. Since we also have $\tb(\mathcal{K})\geq r+1$, we can apply Proposition \ref{prop:preshakeslicebennequin}. The result follows by definition of $r$--shake genus. \end{proof}

\begin{proof}[Proof of Corollary \ref{cor:notrshakeslice}]Choose a Legendrian representative $\mathcal{K}$ of $K$ such that $\tb(\mathcal{K})=\TB(K)$; then, $r\leq \tb(\mathcal{K})-1$ by hypothesis. By Corollary~\ref{cor:shakeslicebennequin}, 
$$1\leq\TB(K)=\tb(\mathcal{K})\leq 2\gsh^r(K)-1$$ 
as needed.\end{proof}

\begin{rem}\label{rem:slicebennequinexpansions}Since $\gsh^0(K)\leq \g_4(K)$ for all knots $K$, we also have an expanded shake slice--Bennequin inequality (Corollary~\ref{cor:shakeslicebennequin})
$$\tb(\mathcal{K})+\lvert\rot(\mathcal{K})\rvert\leq 2\gsh^0(K)-1\leq 2\g_4(K)-1$$
for any Legendrian representative $\mathcal{K}$ of a knot $K$ with $\tb(\mathcal{K})\geq 1$. We saw earlier in~(\ref{eqn:slicebennequinexpansions1}) that there are two other such `expanded' versions of the slice--Bennequin inequality, as follows
$$\tb(\mathcal{K})+\lvert\rot(\mathcal{K})\rvert\leq 2\tau(K)-1\leq 2\g_4(K)-1$$
$$\tb(\mathcal{K})+\lvert\rot(\mathcal{K})\rvert\leq \s(K)-1\leq 2\g_4(K)-1$$
 using \cite[Corollary 1.1]{KronMrow13}\cite[Theorem 1.1]{OzSz03}\cite{Plam04, Shuma07}. It would be interesting to determine the relationships between $\gsh^0(K)$ and $\tau(K)$ or $\s(K)$; no relationship is currently known. \end{rem}

\begin{example}The shake slice--Benequin inequality and related statements allow us to compute the shake genera of several knots. For example, the positive torus knot $T_{p,q}$, with $p,q \geq 2$ and relatively prime, has a Legendrian diagram $\mathcal{T}_{p,q}$ with $\tb(\mathcal{T}_{p,q})=(p-1)(q-1)-1$ and $\rot(\mathcal{T}_{p,q})=0$. We see that $\tb(\mathcal{T}_{p,q})\geq 1$ for all relatively prime $p,q \geq 2$, and therefore, 
$$(p-1)(q-1)-1\leq 2\gsh^r(T_{p,q})-1,$$
for $r\leq(p-1)(q-1)-2$. Since $\g_4(T_{p,q})=\frac{(p-1)(q-1)}{2}$ by~\cite{KronMow93}, we see that 
$$\gsh^0(T_{p,q})=\frac{(p-1)(q-1)}{2}$$
for relatively prime $p,q\geq 2$ as well. In particular, 
\begin{align*}
\gsh^0(T_{2,3})&=1,\\
\gsh^0(T_{2,5})&=2,\\
\gsh^0(T_{2,7})&=3,
\end{align*}
etc.

For relatively prime $p,q\geq 3$, we see that $\TB(T_{p,q})\geq\tb(\mathcal{T}_{p,q})=(p-1)(q-1)-1\geq 3>1$, and therefore, by Corollary~\ref{cor:notrshakeslice}, $\gsh^r(T_{p,q})\geq \frac{1}{2}(\TB(T_{p,q})+1)$, for all $r<\TB(T_{p,q})$. But since $\TB(T_{p,q})\geq (p-1)(q-1)-1$ and $\g_4(T_{p,q})=\frac{(p-1)(q-1)}{2}$ as before~\cite{KronMow93}, we see that for any relatively prime $p,q\geq 3$ and any $r<(p-1)(q-1)-1$,
$$\gsh^r(T_{p,q})=\g_4(T_{p,q})=\frac{(p-1)(q-1)}{2}.$$
 \end{example}

\begin{lem}\label{lem:rshakegenusequality}If $K$ is $(r+1)$--suitable, $\gsh^r(K)=\g_4(K)$. \end{lem}

\begin{proof}By definition, $K$ has a Legendrian representative $\mathcal{K}$ such that $\tb(\mathcal{K})= r+1$ and $\rot(\mathcal{K})=2\g_4(K)-1-r-1$. Note that since $\tb(\mathcal{K})-1= r$, we can use Corollary~\ref{cor:shakeslicebennequin} to see that  
$$\tb(\mathcal{K})+\lvert\rot(\mathcal{K})\rvert\leq 2\gsh^r(K)-1.$$
But we know by~\eqref{eqn:rnotbig} that $2\g_4(K)-1 \geq r+1$ since $K$ is $(r+1)$--suitable. As a result, we see that 
$$r+1+2\g_4(K)-1-r-1\leq 2\gsh^r(K)-1,$$
that is, 
$$\g_4(K)\leq \gsh^r(K).$$
Since $\gsh^r(K)\leq \g_4(K)$ for all knots,
\begin{equation}
\gsh^r(K)=\g_4(K).\qedhere
\end{equation}
\end{proof}

Using Remark~\ref{rem:suitableequalities}, we get the following immediate corollary. 

\begin{cor}If $K$ is $(r+1)$--suitable for $r+1\geq 0$, $K$ is not $r$--shake slice.\end{cor}

\begin{prop}\label{prop:shakegenuscalculation} Let $P$ be any winding number one pattern with a Legendrian diagram $\mathcal{P}$ such that 
$$\tb(\mathcal{P})=1 \text{ and } 0<\g_4(P)\leq \frac{1+\rot(\mathcal{P})}{2}.$$
Then, for any $(r+1)$--suitable knot $K$ and each $i\geq 0$, 
$$\gsh^r(P^i_r(K))=\gsh^r(K)+i\cdot\g_4(P).$$
\end{prop}

\begin{proof} By the case $m=1$ of Lemma~\ref{lem:iteratesuitable}, $P^i_r(K)$ is $(r+1)$--suitable  and $\g_4(P^i_r(K))=\g_4(K)+i\cdot\g_4(P)$ for each $i\geq 0$. By Lemma~\ref{lem:rshakegenusequality},  $\g_4(K)=\gsh^r(K)$ and $\gsh^r(P^i_r(K))=\g_4(P^i_r(K))$ for each $i$. This completes the proof.\end{proof}

\begin{cor}\label{cor:mazurshakegenuscalculation}Let $P$ denote the Mazur pattern. Then for any $(r+1)$--suitable knot~$K$, 
$$\gsh^r(P^i_r(K))=\gsh^r(K)+i,$$
for all $i\geq 0$.
\end{cor}

\begin{proof}Figure \ref{fig:legmazur}(b) shows a Legendrian diagram $\mathcal{P}'$ for the Mazur pattern. Since $\g_4(P)=1$, we see that $\mathcal{P}'$ satisfies the requirements of Proposition \ref{prop:shakegenuscalculation}. The result follows.\end{proof}

\begin{cor}\label{cor:shakegenusnotinvariant}$r$--shake genus is not an invariant of $r$--shake concordance, for any integer $r$. \end{cor}

\begin{proof}Recall from the proof of Theorem \ref{thm:examples} that the knots $\{P^i_r(K) \mid i\geq 0\}$, for the Mazur pattern $P$, are pairwise $r$--shake concordant. Therefore, Corollary~\ref{cor:mazurshakegenuscalculation} shows that $r$--shake genus is not an invariant of $r$--shake concordance, for $r\geq 0$.\end{proof}

Our characterization of $r$--shake slice and $r$--shake concordant knots, Theorem \ref{thm:main}, allows us to find relationships between the $r$--shake genera of a knot and its winding number one satellites as follows. 

\begin{prop}\label{prop:shakegenusincreases}For any knot $K$, any integer $r$, and any winding number one pattern $P$ with $\widetilde{P}$ slice,
$$\gsh^r(K)\leq \gsh^r(P_r(K)).$$\end{prop}
\begin{proof} By Proposition \ref{prop:satelliteimpliesshake1n}  there is a $(1,\,n)$ $r$--shake concordance $C$ from $P_r(K)$ to $K$. Consider a surface $\Sigma$ with $\g(\Sigma)=\gsh^r(P_r(K))$ with boundary an $m$--component $r$--shaking of $P_r(K)$. By gluing on $m$ copies of $C$ (algebraically one) to $\Sigma$, using Remark~\ref{rem:parallelcopies}, we get a surface of genus $\gsh^r(P_r(K))$ bounded by an $mn$--component $r$--shaking of $K$, completing the proof.\end{proof}

We can easily see that the above is true more generally, i.e.\ if $J$ is $(1,\,n)$ $r$--shake concordant to $K$, then $\gsh^r(K)\leq\gsh^r(J)$. Of course, by Proposition~\ref{prop:1mshakeimpliessatellite}, such a $J$ must be concordant to an $r$--twisted satellite of $K$ with companion a winding number one pattern $P$ with $\widetilde{P}$ ribbon. 

Using Proposition \ref{prop:shakegenusincreases} and Corollaries \ref{cor:tauzero} and \ref{cor:arfzero}, we see the following. 

\begin{cor}\label{cor:taunochance} Fix a knot $K$. If $\tau(K)\neq 0$ then $P(K)$ is not 0--shake slice (and therefore, is not slice), for \textit{any} winding number one pattern $P$ with $\widetilde{P}$ slice. 

If $\Arf(K)\neq 0$ then $P_r(K)$ is not $r$--shake slice (and therefore, is not slice) for \textit{any} $r$ and for \textit{any} winding number one pattern $P$ with $\widetilde{P}$ slice. \end{cor} 

\begin{proof}Since $\tau(K)\neq 0$, $K$ cannot be 0--shake slice by Corollary~\ref{cor:tauzero}, i.e.\ $\gsh^0(K)>0$. By Proposition~\ref{prop:shakegenusincreases}, $\gsh^0(P(K))>0$.  The second statement follows similarly, since if $\Arf(K)\neq 0$, $K$ is not $r$--shake slice, i.e.\ $\gsh^r(K)>0$, for any $r$, by Corollary~\ref{cor:arfzero}. \end{proof}

Indeed, if $P(K)$ is 0--shake slice then it is slice in a homology 4--ball by Proposition~\ref{prop:shakeconchomcob}. Then, by ~\cite[Corollary 4.2 ($n=1$)]{CDR14}, $K$ is also slice in a homology ball. 

\begin{prop}\label{prop:shakedecreaseslice}Fix an integer $r$, and a knot $K$. There exists a winding number one pattern $P$ with $\widetilde{P}$ slice, such that $\g_4(P_r(K)) < \g_4(K)$ if and only if $\gsh^r(K)<\g_4(K)$.  Equivalently, $\gsh^r(K)=\g_4(K)$ if and only if $\g_4(P_r(K))\geq\g_4(K)$ for all winding number one operators $P$ with $\widetilde{P}$ slice.\end{prop}

\begin{proof}The forward direction follows immediately from Proposition~\ref{prop:shakegenusincreases} since $\gsh^r(J)\leq\g_4(J)$ for all knots $J$. 

For the backwards direction, consider a surface $\Sigma\subseteq B^4$ bounded by an $m$--component $r$--shaking of $K$ with genus $\gsh^r(K)$. By a small isotopy we can assume that the radial function on $B^4$ is Morse when restricted to $\Sigma$. We can then further assume that  all the local maxima and all the ``join'' saddles occur in an $\epsilon$ collar of $\partial B^4$ so that $\Sigma\cap (S^3\times\{\epsilon\})$ is a connected 1--manifold $J$, and $\Sigma\cap (S^3\times [0,\epsilon])$ is a $(1,\,m)$ $r$--shake concordance from $J$ to $K$. By Proposition \ref{prop:1mshakeimpliessatellite}, $J$ is concordant to $P_r(K)$ for some winding number one pattern $P$ with $\widetilde{P}$ slice. This shows that 
\begin{equation}\label{eq:inproof}
\g_4(P_r(K))=\g_4(J)\leq\g(\Sigma)=\gsh^r(K) < \g_4(K). 
\end{equation}
Moreover since $\g_4(P_r(K))< \g_4(K)$, $P$ is not the trivial pattern even modulo concordance in $ST\times [0,1]$. \end{proof}

\begin{cor}Fix an integer $r$, and a knot $K$. If $\gsh^r(K)<\g_4(K)$ then there exists a  pattern $P$, \textit{non-trivial even modulo concordance}, with winding number one and $\widetilde{P}$ slice such that $\gsh^r(P_r(K))=\g_4(P_r(K))=\gsh^r(K)$. \end{cor}
\begin{proof} Let $P$ be pattern obtained in the proof of the above proposition. We have
$$\gsh^r(K)\leq \gsh^r(P_r(K))\leq g_4(P_r(K))\leq \gsh^r(K),$$
using Proposition \ref{prop:shakegenusincreases} for the first inequality and ~\eqref{eq:inproof} for the third inequality. Thus the inequalities are equalities. $P$ is non-trivial even modulo concordance by the last line of the previous proof.
\end{proof}

%======================================================
\section{Examples of $r$--shake slice knots}\label{sec:oldexamples}

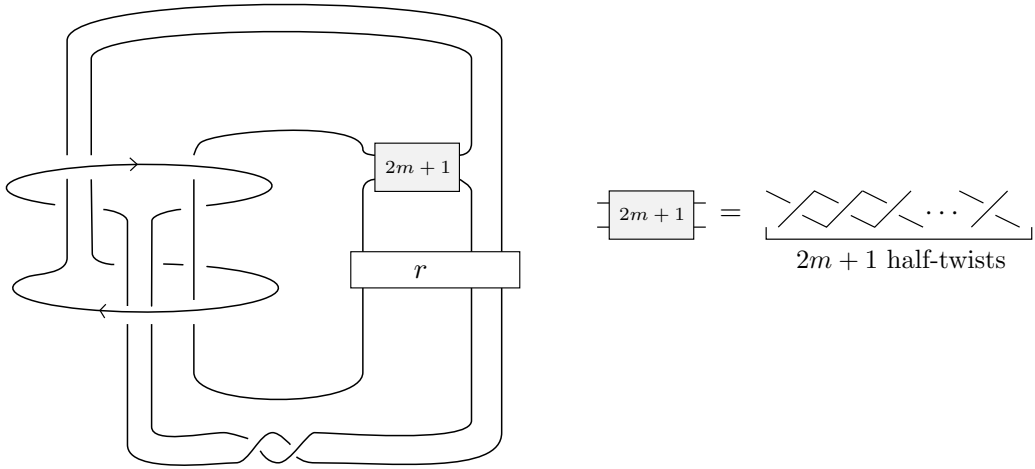
\begin{figure}[t]
\includegraphics[width=5.5in]{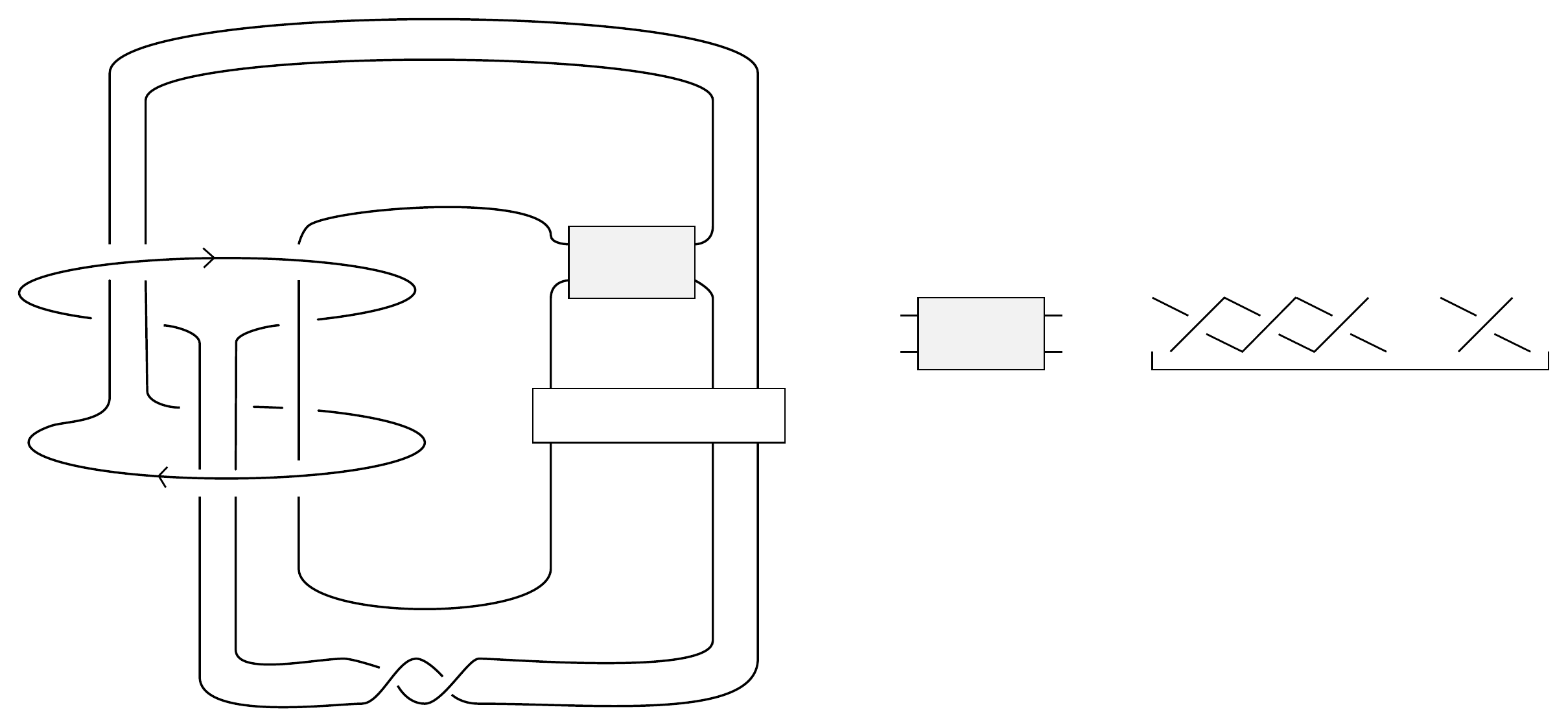}
\put(-3.3,1.05){$r$}
\put(-3.45,1.6){\tiny $2m +1$}
\put(-2.225,1.35){\tiny $2m +1$}
\put(-1.7, 1.35){=}
\put(-1.3,1.1){\small $2m+1$ half-twists}
\put(-0.625,1.35){$\cdots$}
\caption{The knots $K_{r,\,m}$ of Akbulut~\cite{Ak77} and Abe--Jong--Omae--Takeuchi~\cite{AbeJongOmaeTake13}. The strands passing through the box containing $2m+1$ should be given $2m+1$ left-handed half twists as shown above.}\label{fig:prevexamples}
\end{figure}

As we mentioned in Section 1, there are examples of $r$--shake slice knots, for $r\neq 1$, due to Akbulut~\cite{Ak77,Ak93} and Abe--Jong--Omae--Takeuchi~\cite{AbeJongOmaeTake13}. The latter's examples are shown in Figure~\ref{fig:prevexamples} (these examples generalize the ones from~\cite{Ak93}, which are in turn generalizations of the 1--shake slice example in~\cite{Ak77}; the fact that Akbulut's 1--shake slice example in~\cite{Ak77} is of this form was also shown by Lickorish in~\cite{Lic79}). They showed that for any $m\geq 0$ and $r\neq 0$, the knot $K_{r,\,m}$ is $r$--shake slice. In~\cite{Lic79}, Lickorish gave an alternate proof of why $K_{1,\,0}$ is 1--shake slice; in fact, his explanation can be easily modified to apply to $K_{r,\,m}$ for all $r\neq 0$ and $m\geq 0$, as he himself asserted in~\cite[Remarks 2 and 3]{Lic79}. 

Since the knots $K_{r,\,m}$ are $r$--shake slice, there must exist winding number one patterns $P^{(m)}$ with $\widetilde{P^{(m)}}$ slice, for which $P^{(m)}_r(K_{r,\,m})$ is a slice knot, by Corollary~\ref{cor:characterizeshakeslice}. However, our proof for Corollary~\ref{cor:characterizeshakeslice} does not give an explicit construction of such a $P^{(m)}$. Below we show that we can explicitly construct such satellite operators, modulo the smooth 4--dimensional Poincar\'{e} Conjecture.

\begin{figure}[b]
\includegraphics[width=\textwidth]{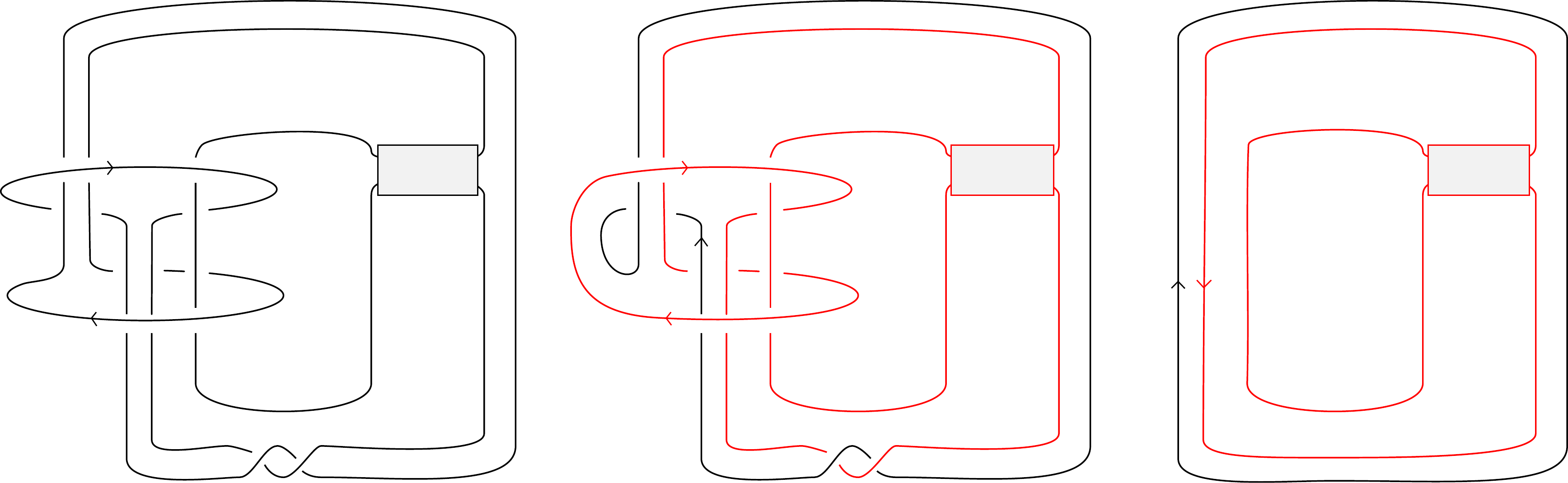}
\put(-0.475,1.075){\tiny $2m +1$}
\put(-2.15,1.075){\tiny $2m +1$}
\put(-4.1725,1.075){\tiny $2m +1$}
\caption{The knot $K_{0,\,m}$ is ribbon. Left: the knot $K_{0,\,m}$. Center: the result of attaching a band (a ribbon move). Right: a further isotopy shows a 2--component unlink.}\label{fig:K0mribbon}
\end{figure}

Firstly, note that the knot $K_{0,\,m}$ is ribbon, as shown in Figure~\ref{fig:K0mribbon}. Consider the family of winding number one patterns $R^{(m)}$, $m\geq 0$, shown in Figure \ref{fig:prevexamplepattern}, where $\widetilde{R^{(m)}}$ is clearly $K_{0,\,m}$. Then it is easy to see that $R^{(m)}_r(U)=K_{r,\,m}$, where $U$ is the unknot. Therefore, we are seeking patterns $P^{(m)}$ such that $P^{(m)}_r(R^{(m)}_r(U))$ is concordant to $U$ (i.e.\ is slice.). 

\begin{figure}[t]
\includegraphics[width=4in]{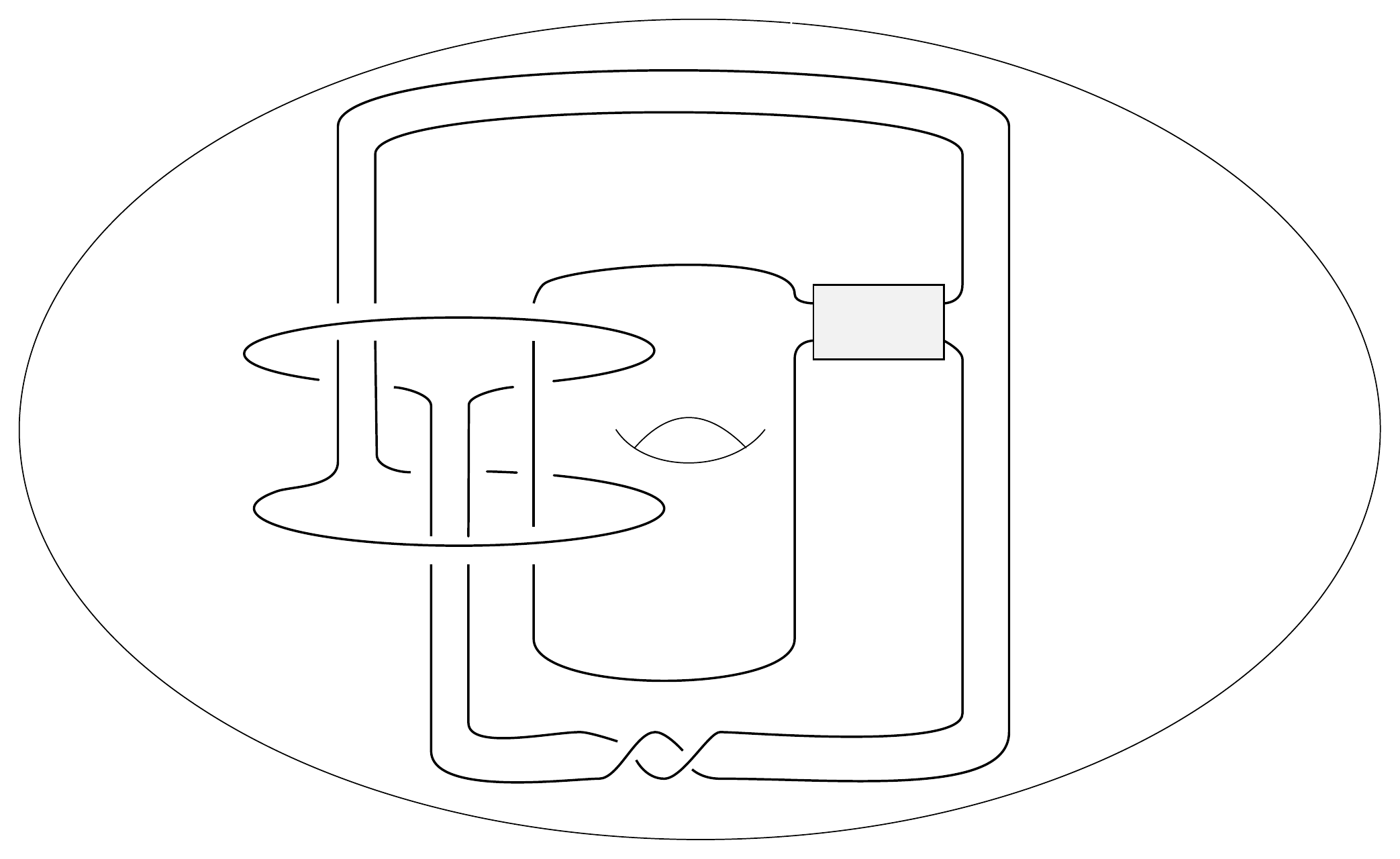}
\put(-1.66,1.5){\tiny $2m +1$}
\caption{A family of winding number one patterns $R^{(m)}$, $m\geq 0$, with $\widetilde{R^{(m)}}=K_{0,\,m}$. }\label{fig:prevexamplepattern}
\end{figure}

Our main tool will be the results of~\cite[Section 3]{DR13}, for which we recall some notions from~\cite{DR13}. Recall that there is a well-defined notion of composing two patterns $P$ and $Q$, contained in standard solid tori $V_P$ and $V_Q$: loosely speaking, we drill out a regular neighborhood of $Q$ in $V_Q$ and glue in $V_P$ in an untwisted manner; the image of $P$ in this new manifold, which can be seen to be a solid torus (denoted $V_{P\star Q}$), is the composed pattern $P\star Q$. The set of isotopy classes of patterns forms a monoid under this operation, and moreover, this composition has the handy property that $(P\star Q)(K)=P(Q(K))$ for all knots $K$, i.e.\ the classical untwisted satellite operation is a monoid action by the monoid of isotopy classes of patterns on the set of isotopy classes of knots. (Further details can be found in~\cite[Section 2]{DR13}, as well as the proof of the following proposition.) In fact, such a relationship is sometimes true for twisted satellites as well, as we see in the following proposition; we postpone the proof to the end of this section.

\begin{prop}\label{prop:iteratesorcompose}Let $P$ and $Q$ be patterns with winding number $w(P)$ and $w(Q)$ respectively, and $r$ an integer. For any knot $K$, the iterated twisted satellite $P_r(Q_r(K))$ is isotopic to the twisted satellite $(P\star Q)_r(K)$ if and only if $w(Q)=\pm 1$ or $r=0$. \end{prop}

Note that the patterns $R^{(m)}$  shown in Figure~\ref{fig:prevexamplepattern} have winding number one. Moreover, these patterns actually have \textit{inverses} by the following theorem from~\cite{DR13}. 

\begin{thm}[Theorem 3.4 of~\cite{DR13}]\label{thm:honestinverse}Let $P$ be a winding number one pattern contained in a solid torus $V$. If the meridian of $P$ is in the subgroup of $\pi_1(V-N(P))$ normally generated by the meridian of $V$ then there exists another winding number one pattern $\overline{P}$ such that the composed pattern $\overline{P}\star P$ is concordant  to the trivial pattern (namely the core of a solid torus) in a possibly exotic copy of $S^1\times D^2 \times [0,1]$. \end{thm}

Let $V_{R^{(m)}}$ denote the solid torus containing the pattern $R^{(m)}$. To ensure that each pattern $R^{(m)}$ satisfies the requirements of the above theorem, it suffices to show that the meridian of $R^{(m)}$ is nullhomotopic in the 3--manifold $N$ obtained from $V_{R^{(m)}}-N(R^{(m)})$ by adding a 2--handle to the meridian of $V_{R^{(m)}}$. The result of sliding $R^{(m)}$ over this 2--handle twice (isotopies in $N$) is depicted in  Figure~\ref{fig:proofofsurjex}. In the result of the isotopy, the meridian of $R^{(m)}$ cobounds an annulus with the meridian of $V_{R^{(m)}}$ and so bounds a disk in $N$.  

\begin{figure}[t]
        \centering
        \begin{subfigure}[t]{0.32\textwidth}
        \centering
		\includegraphics[width=1.55in]{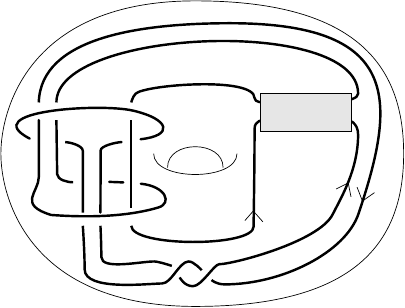}
		\put(-0.55,0.72){\tiny $2m+1$}
               	 \caption*{The satellite operator $R_m$}
        \end{subfigure}%
        \hspace{5pt}
        \begin{subfigure}[t]{0.32\textwidth}
        	      \centering
	      \includegraphics[width=1.55in]{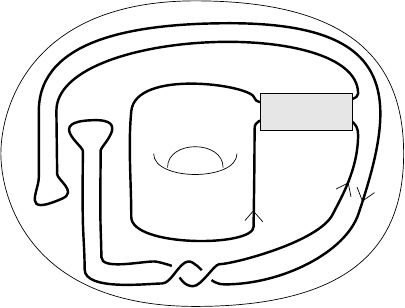}
	      \put(-0.55,0.72){\tiny $2m+1$}
               \caption*{The result of sliding $R_m$ over the meridian of $V_{R^{(m)}}$}
        \end{subfigure}%
        \hspace{5pt}
        \begin{subfigure}[t]{0.32\textwidth}
        	      \centering
	     \includegraphics[width=1.55in]{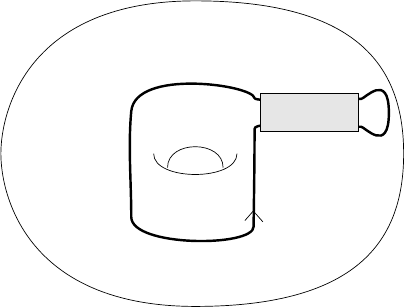}
	      \put(-0.55,0.72){\tiny $2m+1$}
               \caption*{The result of a further isotopy}
        \end{subfigure}%
        \caption{The patterns $R_m$ from Figure~\ref{fig:prevexamplepattern} satisfy the requirements of Theorem~\ref{thm:honestinverse}}\label{fig:proofofsurjex}
\end{figure}

Then, by using Theorem~\ref{thm:honestinverse}, there exist winding number one patterns $P^{(m)}$ such that $(P^{(m)}\star R^{(m)})_r(K)$ is concordant to $K$ in a possibly exotic $S^3\times [0,1]$ for any knot $K$ and integer $r$, since the trivial pattern (even if twisted) acts trivially on the set of knot concordance classes. However, by Proposition~\ref{prop:iteratesorcompose} we know that $(P^{(m)}\star R^{(m)})_r(K)$ is also isotopic to $P^{(m)}_r(R^{(m)}_r(K))$ for any knot $K$, since each $R^{(m)}$ is winding number one. In particular, this shows that for $U$ the unknot and $K_{r,\,m}$ the knots of Akbulut and Abe--Jong--Omae--Takeuchi, 
$$P^{(m)}_r(K_{r,\,m}) = P^{(m)}_r(R^{(m)}_r(U)) = (P^{(m)}\star R^{(m)})_r(U)$$ 
where the last is known to be concordant to $U$, in a possibly exotic $S^3\times [0,1]$, i.e.\ is slice in a possibly exotic $B^4$.  

We now check that $\widetilde{P^{(m)}}$ slice, for all $m\geq 0$. Recall that $\widetilde{R^{(m)}}=R^{(m)}(U)$ is slice, and as a result, $\widetilde{P^{(m)}}=P^{(m)}(U)$ is concordant to $P^{(m)}(R^{(m)}(U))$ which we know to be isotopic to $(P^{(m)}\star R^{(m)})(U)$. By Theorem~\ref{thm:honestinverse} $(P^{(m)}\star R^{(m)})(U)$ is slice, in a possibly exotic $B^4$, for all $m\geq 1$.

\begin{rem}We can explicitly draw the patterns $P^{(m)}$ used above. This is shown in Remark 3.6 and Figure 8 of~\cite{DR13}.\end{rem}

\begin{rem}If we assume the smooth 4--dimensional Poincar\'{e} Conjecture, then any manifold $\mathcal{B}$ which is a possibly exotic copy of $B^4$ is in fact diffeomorphic to $B^4$ -- see~\cite[Proof of Proposition~3.2]{CDR14} for a proof.\end{rem}

\begin{rem} Note that Akbulut's example of a 2--shake slice knot that is not slice (from~\cite{Ak77}) does not appear to belong to the family shown in Figure~\ref{fig:prevexamples}, at least at first glance. \end{rem}

\subsection{New examples of shake slice knots}
Theorem~\ref{thm:honestinverse}, along with Proposition~\ref{prop:iteratesorcompose} and our characterization theorem, Theorem~\ref{thm:main}, also gives us a way to construct possible new examples of shake slice knots, as follows. Let $P$ be a pattern, with $\widetilde{P}$ slice, that satisfies the conditions of Theorem~\ref{thm:honestinverse}. Then, as before, modulo the smooth 4--dimensional Poincar\'{e} Conjecture, there exists another winding number one pattern $\overline{P}$ such that $(\overline{P}\star P)_r(K)$ is concordant to $K$ for any knot $K$ and any integer $r$. Then consider the knot $P_r(U)$ for any integer $r$ and $U$ the unknot. We have that $\overline{P}_r(P_r(U))$ is concordant to $(\overline{P}\star P)_r(U)$, which we know is concordant to $U$, that is, is slice. Moreover, $\widetilde{\overline{P}}$ is slice, since $\widetilde{\overline{P}}=\overline{P}(U)$ is concordant to $\overline{P}(P(U))$, which in turn is concordant to $(\overline{P}\star P)(U)$, which we know is slice as before. Then, by our characterization theorem, Theorem~\ref{thm:main}, we see that the knot $P_r(U)$ is $r$--shake slice. 

Therefore, we have proved the following proposition. 

\begin{prop}\label{prop:newshakeslice}Let $P$ be a winding number one pattern in a solid torus $V$, such that $\widetilde{P}$ is slice, and the meridian of $P$ is in the subgroup of $\pi_1(V-N(P))$ normally generated by the meridian of $V$. Then for any integer $r$, the knot $P_r(U)$ is $r$--shake slice, where $U$ is the unknot, modulo the smooth 4--dimensional Poincar\'{e} Conjecture.\end{prop}

\begin{rem}Unfortunately, the above proposition does not yield any new examples of 0--shake slice knots since if $r=0$, $P_r(U)$ is just the knot $\widetilde{P}$ which is slice by hypothesis.\end{rem}

\begin{figure}[b]
\includegraphics{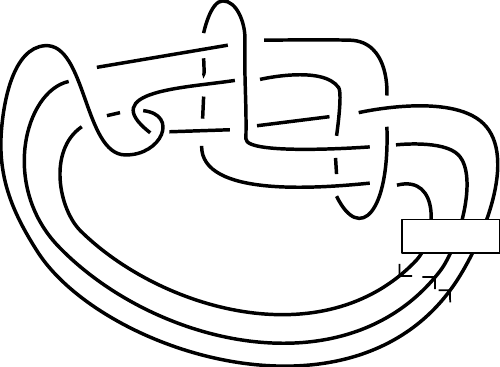}
\put(-0.25,0.48){$r$}
\caption{Here $r$ is any integer. For any fixed value of $r$, the knot pictured is $r$--shake slice.}\label{fig:newshakeslice}
\end{figure} 

For example, by using Proposition~\ref{prop:newshakeslice}, we can see that the knots given in Figure~\ref{fig:newshakeslice} are $r$--shake slice in a possibly exotic $B^4$. To do so, we need to verify that the winding number one pattern $P$ corresponding to these knots, shown in Figure~\ref{fig:verifynewshakeslice}(a), satisfies the conditions given in Proposition~\ref{prop:newshakeslice}. That $\widetilde{P}$ is slice can be seen by attaching bands; we leave this to the reader. That the meridian of $P$ is in the subgroup of $\pi_1(V-N(P))$ normally generated by the meridian of $V$, the solid torus containing $P$, can be seen via the pictures in Figure~\ref{fig:verifynewshakeslice}, which show that the meridian of $P$ bounds a disk in the manifold obtained from $V-N(P)$ by attaching a 2--handle along the meridian of $V$.

\begin{figure}[t]
\includegraphics{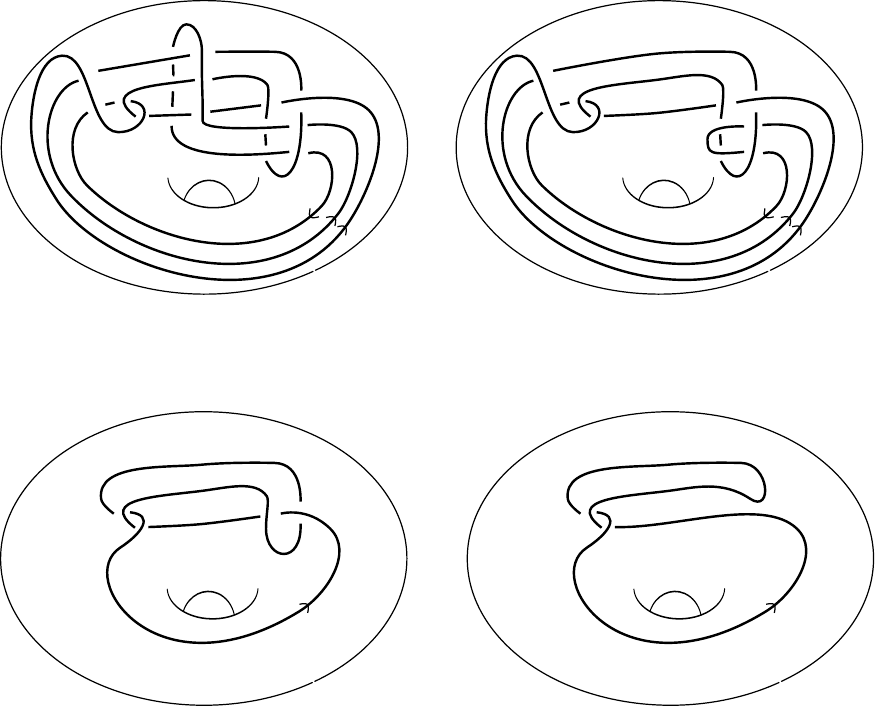}
\put(-2.75,1.45){(a)}
\put(-0.95,1.45){(b)}
\put(-2.75,-0.2){(c)}
\put(-0.95,-0.2){(d)}
\caption{Verifying that the knots in Figure~\ref{fig:newshakeslice} are shake slice. (a): The pattern $P$ corresponding to the knots in Figure~\ref{fig:newshakeslice}. (b): The result of sliding $P$ once over the meridian of $V$. (c): The result of an isotopy. (d): By sliding $P$ over the meridian of $V$ again, we obtain the core of the solid torus.}\label{fig:verifynewshakeslice}
\end{figure}
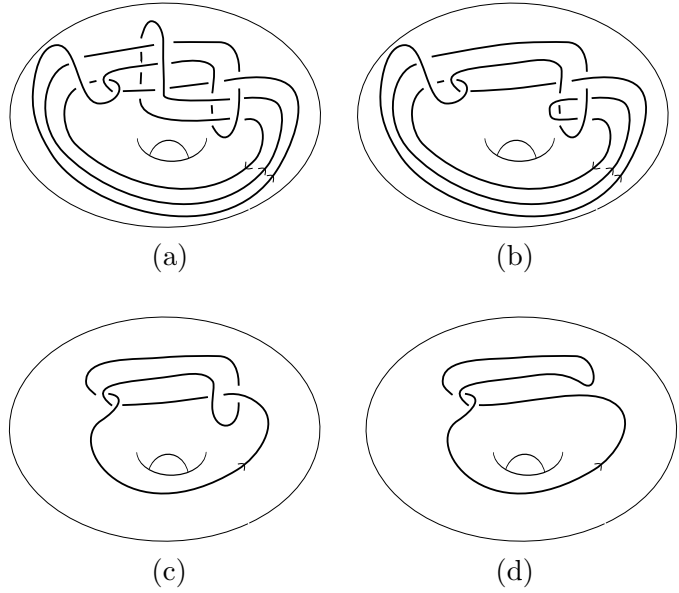

To check whether these new examples are slice, one could compute various knot concordance invariants. The previous examples of Akbulut and Abe--Jong--Omae--Takeuchi were shown to be non-slice using signature and Alexander polynomials. 

In summary, to construct new families of knots that are shake slice in a possibly exotic $B^4$, it suffices to produce patterns satisfying the requirements of Proposition~\ref{prop:newshakeslice}.

\begin{proof}[Proof of Proposition~\ref{prop:iteratesorcompose}]Let $V_P$ and $V_Q$ denote the solid tori containing $P$ and $Q$ respectively. Let $E(P):=V_P-N(P)$ and $E(Q):=V_Q-N(Q)$, the exteriors of the patterns. The manifold $E(P)$ has two toral boundary components, on which there are four curves of interest: $m_i(P)$, the meridian of $P$, $\ell_i(P)$, the longitude of $P$, $m_o(P)$, the meridian of $V_P$, and $\ell_o(P)$, the longitude of $V_P$, oriented such that $m_o(P)=w(P)m_i(P)$, $\ell_i(P)=w(P)\ell_o(P)$, and $\lk(m_i(P),\ell_i(P))=\lk(m_o(P),\ell_o(P))=1$. Similarly, we have curves $m_i(Q)$, $\ell_i(Q)$, $m_o(Q)$, and $\ell_o(Q)$. 

A knot $K$ is determined by its exterior $E(K):=S^3-N(K)$ along with an oriented longitude of the boundary torus of $E(K)$. By definition of the satellite construction, the exterior of $Q_r(K)$, $E(Q_r(K))$, is built from $E(Q)$ and $E(K)$ by attaching them along 
$$m_o(Q)  \sim \mu(K)\text{ and }\ell_o(Q) -r\cdot m_o(Q) \sim  \lambda(K), $$
where $\mu(K)$ and $\lambda(K)$ denote the meridian and untwisted longitude of $K$; the untwisted longitude for $Q_r(K)$ can be seen to be 
$$\ell_i(Q)-r\cdot w(Q)^2m_i(Q)$$
since it must be null homologous in $E(Q_r(K))$. 
%That is, 
%$$\lambda(Q_r(K))=\ell_i(Q)-r\cdot w(Q)^2m_i(Q)\text{ and }\mu(Q_r(K))=m_i(Q).$$ 
Repeat this process to construct $E(P_r(Q_r(K)))$ from $E(Q_r(K))$ and $E(P)$, by gluing 
$$m_o(P)  \sim \mu(Q_r(K))=m_i(Q)$$
$$ \text{ and }\ell_o(P) -r\cdot m_o(P) \sim  \lambda(Q_r(K))=\ell_i(Q)-r\cdot w(Q)^2m_i(Q).$$
This yields the 3--manifold $E(P_r(Q_r(K)))$ with the untwisted longitude 
$$\lambda(P_r(Q_r(K)))=\ell_i(P)-r\cdot w(P)^2m_i(P).$$

On the other hand, to construct $E(P\star Q)$, we glue together $E(P)$ and $E(Q)$ by identifying 
$$m_o(P)\sim m_i(Q)\text{ and }\ell_o(P)\sim\ell_i(Q).$$
The resulting 3--manifold is a new solid torus, denoted $V_{P\star Q}$, where we can see that 
\begin{align*}
m_o(P\star Q)=m_o(Q),\,\, &\ell_o(P\star Q)=\ell_o(Q)\\
m_i(P\star Q)=m_i(P),\,\, &\ell_i(P\star Q)=\ell_i(P)
\end{align*}
Note also that $w(P\star Q)=w(P)\cdot w(Q)$. By gluing $E(P\star Q)$ to $E(K)$, via the identifications 
$$m_o(P\star Q)\sim\mu(K)\text{ and }\ell_o(P\star Q)\sim\lambda(K),$$
we obtain the manifold $E((P\star Q)_r(K))$ with untwisted longitude 
\begin{align*}
\lambda(E((P\star Q)_r(K)))&=\ell_i(P\star Q)-r\cdot w(P\star Q)^2m_i(P\star Q)\\
&=\ell_i(P)-r\cdot w(P)^2w(Q)^2m_i(P).
\end{align*}

We now show that the manifolds $E((P\star Q)_r(K))$ and $E(P_r(Q_r(K)))$ are the same if and only if $w(Q)=\pm 1$. This will complete the proof since we already see that the the untwisted longitudes $\lambda((P\star Q)_r(K))$ and $\lambda(P_r(Q_r(K)))$ are the same if and only if $w(Q)=\pm 1$ or $r=0$. Certainly the manifolds $E((P\star Q)_r(K))$ and $E(P_r(Q_r(K)))$ are built using the same pieces. Therefore, we only need to verify that the gluing maps are the same if and only if $w(Q)=\pm 1$. It is clear that $E(Q)$ is attached to $E(K)$ in the same way in both manifolds, since $m_o(P\star Q)=m_o(Q)$ and $\ell_o(P\star Q)=\ell_o(Q)$. We see that $E(P)$ and $E(Q)$ are also glued the same way in both cases since the maps identify 
$$m_o(P)\sim m_i(Q)$$
$$\text{ and } \ell_o(P) -r\cdot m_o(P) \sim \ell_i(Q)-r\cdot w(Q)^2m_i(Q)$$
 in one case, and 
 $$m_o(P)\sim m_i(Q)$$
 $$\text{ and }\ell_o(P)\sim\ell_i(Q)$$
  in the other. Identifying $\ell_o(P)\sim \ell_i(Q)$ is the same as identifying $\ell_o(P) -r\cdot m_o(P) \sim \ell_i(Q)-r\cdot w(Q)^2m_i(Q)$ if and only if $w(Q)^2=1$ or $r=0$ since $m_i(Q)\sim m_o(P)$ in both manifolds. \end{proof}
  
In the fact, the argument in the above proof can be easily modified to show the following more general fact.

\begin{prop}Let $P$ and $Q$ be patterns with winding number $w(P)$ and $w(Q)$ respectively, and let $P\star Q$ denote their composition as patterns, and $r,s$ be integers. The iterated twisted satellite $P_s(Q_r(K))$ is equal to the twisted satellite $(P_{s-r}\star Q)_r(K)$ if and only if $w(Q)=\pm 1$ or $r=0$.\end{prop}

%What remains to be done:
%\begin{enumerate}
%\item We haven't mentioned anything so far about `strong' shake concordance and `strong' shake sliceness. \textcolor{red}{Why should we? Hmm maybe we should}
%\item Should we mention somewhere that `shake slice' is not preserved under `shake concordance'?
%\item Is there any reason that our $r$--shake examples are better than the previously known ones for $r\neq 0$? 
%\end{enumerate}

\bibliographystyle{alpha}
\bibliography{bib}

\end{document}